\documentclass[final]{siamltex}
\usepackage{amsmath,amsfonts,amssymb}
\usepackage{latexsym}
\usepackage{epsfig,overpic}
\usepackage{subfigure}
\usepackage{color}

\let\ep\varepsilon

\newcommand{\ga}{\gamma}

\newcommand{\bn}{\mathbf n}

\newcommand{\bw}{\mathbf w}

\newcommand{\T}{\mathcal T}

\newcommand{\Div}{\mathop{\rm div}}
\newcommand{\DivG}{{\,\operatorname{div_\Gamma}}}
\newcommand{\nablaG}{\nabla_\Gamma}

\newcommand{\Gs}{\mathcal{S}} 
\DeclareGraphicsExtensions{.pdf,.eps,.ps,.eps.gz,.ps.gz,.eps.Y}

\newcommand{\la}{\left\langle}
\newcommand{\ra}{\right\rangle}
\newcommand{\Wo}{\overset{\circ}{W}}

\newtheorem{remark}{Remark}[section]
\def\enorm#1{|\!|\!| #1 |\!|\!|}
\begin{document}
\title{An Eulerian space-time finite element method for diffusion problems on evolving surfaces}
\author{Maxim A. Olshanskii\thanks{Department of Mathematics, University of Houston, Houston, Texas 77204-3008
(molshan@math.uh.edu) and Dept. of Mechanics and Mathematics, Moscow State University, Russia.}
\and Arnold Reusken\thanks{Institut f\"ur Geometrie und Praktische  Mathematik, RWTH-Aachen
University, D-52056 Aachen, Germany (reusken@igpm.rwth-aachen.de,xu@igpm.rwth-aachen.de).}
\and Xianmin Xu$^\dag$\thanks{LSEC, Institute of Computational
  Mathematics and Scientific/Engineering Computing,
  NCMIS, AMSS, Chinese Academy of Sciences, Beijing 100190, China (xmxu@lsec.cc.ac.cn).}
}
\maketitle
\begin{abstract} In this paper, we study numerical methods for the solution of partial differential equations on evolving surfaces. The evolving hypersurface in $\Bbb{R}^d$ defines a $d$-dimensional space-time manifold in the space-time continuum $\Bbb{R}^{d+1}$.
 We derive and analyze  a variational formulation for a class of diffusion problems on the space-time manifold. For this variational formulation new well-posedness and stability results are derived.
The analysis is based on an inf-sup condition and involves some natural, but non-standard, (anisotropic) function spaces.  Based on this formulation  a discrete in time variational formulation is introduced that is very suitable as a starting point for a discontinuous Galerkin (DG) space-time finite element discretization. This DG space-time method is explained and results of numerical experiments are presented that illustrate  its properties.
\end{abstract}
%
\section[Introduction]{Introduction}\label{sec:introduction}
Partial differential equations (PDEs) posed on evolving surfaces arise in many applications.
In fluid dynamics, the concentration of surface active agents attached to an interface
between two phases of immiscible fluids is governed by a transport-diffusion  equation
on the interface~\cite{GReusken2011}. Another example is the diffusion of  trans-membrane receptors
in the membrane of a deforming and moving cell, which is typically modeled
by a parabolic PDE posed on an evolving surface \cite{Alberta_etal}.

Recently, several approaches   for solving PDEs on evolving surfaces numerically have been introduced.
The finite element method of Dziuk  and Elliott~\cite{Dziuk07}  is based on the \emph{Lagrangian} description
of a surface evolution and benefits from a special invariance property of test functions along
material trajectories. If one considers the \emph{Eulerian} description of a surface evolution, e.g., based  on the level set method~\cite{Sethian96AN}, then the surface is usually defined implicitly. In this case, regular surface triangulations and material trajectories of points on the surface are not easily  available.
Hence, Eulerian numerical techniques for the discretization of PDEs on surfaces have been studied in the literature.
In \cite{AS03,XuZh}   numerical approaches were introduced that are based on extensions of PDEs off a two-dimensional surface to a three-dimensional  neighbourhood of the surface. Then one can apply a standard finite element or (as was done in \cite{AS03,XuZh}) finite difference disretization to treat the extended equation in $\mathbb{R}^3$. The  extension, however, leads  to
degenerate parabolic PDEs and requires the solution of equations in a higher dimensional domain. For a detailed discussion of this extension approach we refer to \cite{Greer,Dziukelliot2x,ChOlsh}. A related approach was developed in ~\cite{ElliottStinner}, where advection-diffusion
 equations are numerically solved on evolving diffuse interfaces.

A different  Eulerian technique for the numerical solution of an elliptic PDE
posed on a hypersurface in $\mathbb{R}^3$ was introduced in \cite{OlshReusken08,OlshanskiiReusken08}.
The main idea of this method is to use finite element spaces that are induced by the volume triangulations (tetrahedral decompositions) of a bulk domain in order to discretize a partial differential equation on the embedded surface. This method does not use an extension of the surface partial differential equation.  It is instead based on a restriction (trace) of the outer  finite element spaces to the (approximated) surface. This leads to discrete problems for which the number of degrees of freedom corresponds to the two-dimensional nature of the surface problem, similar to the Lagrangian approach. At the same time, the method is essentially Eulerian as the surface is not tracked by a surface mesh and may be defined
implicitly as the zero level of a level set function. For the discretization of the PDE on the surface, this zero level then has to be reconstructed. Optimal discretization error bounds were proved in  \cite{OlshReusken08}.
The approach was further developed in \cite{DemlowOlshanskii12,ORX}, where adaptive and streamline diffusion variants of this surface finite element were introduced and analysed. These papers \cite{OlshReusken08,OlshanskiiReusken08,DemlowOlshanskii12,ORX}, however,
treated elliptic and parabolic equations on \emph{stationary} surfaces.

The goal of this paper is to extend  the approach from  \cite{OlshReusken08} to parabolic equations  on \emph{evolving} surfaces. An evolving surface defines
a three-dimensional {space-time manifold} in the space-time continuum $\mathbb{R}^{4}$.
The surface finite element method that we introduce is based on the {traces}  of outer {space-time} finite element functions on this  manifold. The finite element functions are piecewise polynomials with respect to a volume mesh, consisting  of four-dimensional prisms (4D prism = 3D tetrahedron $\times$ time interval).
For this discretization technique, it is natural to start with a variational formulation of the differential problem on the space-time manifold. To our knowledge such a formulation has not been studied in the literature, yet.  One new result of this paper is the derivation and analysis of  a variational formulation for a class of diffusion problems on the space-time manifold. For this  formulation we prove well-posedness and stability results.
The analysis is based on an inf-sup condition and involves some natural, but non-standard, anisotropic function spaces. A second important result is the formulation and analysis  of a discrete in time variational formulation that is very suitable as a starting point for a discontinuous Galerkin space-time finite element discretization.
Further, we present a finite element method, which then results in discretization (in space and time) of a parabolic equation on an evolving surface.

 The discretization approach  based on  traces of an outer space-time finite element space studied here is also investigated in the recent report \cite{refJoerg}. 
 In \cite{refJoerg} there is no analysis of the corresponding continuous variational formulation, which is the main topic of this paper. On the other hand, in \cite{refJoerg} one finds more information on implementation aspects and an extensive numerical study of properties (accuracy and stability) of this method and some of its variants. Further results of numerical experiments for the example of surfactant transport over two colliding spheres can be found in \cite{GOR2014}. 
 We only very briefly comment on implementation aspects and illustrate  accuracy and stability properties of the discretization method by results of  a few numerical experiments.

In this paper, we do not study discretization error bounds for the presented Eulerian space-time finite element method.  This is a topic of current research,  first results of which are presented in the follow-up report \cite{refARrecent}.

The remainder of the paper is organized as follows. In Section~\ref{sec2} we review surface transport-diffusion equations
and introduce a space-time weak formulation.  Some required results for surface functional spaces
are proved in Section~\ref{sec3}.  A general space-time variational formulation and corresponding well-posedness results are presented  in Section~\ref{sec4}. A semi-discrete in time method is analyzed in Section~\ref{sec5}. A fully discrete space-time
finite element method is considered in Section~\ref{sec6}. Section~\ref{sec7} contains results of some numerical
experiments.

\section{Diffusion equation  on an evolving surface}\label{sec2}

Consider a surface $\Gamma(t)$ passively advected by a velocity field $\bw=\bw(x,t)$, i.e. the normal velocity of $\Gamma(t)$ is given by $\bw \cdot \bn$, with
$\bn$ the unit normal on $\Gamma(t)$. We assume that for all $t \in [0,T] $,  $\Gamma(t)$ is a smooth hypersurface that is  closed ($\partial \Gamma =\emptyset$), connected, oriented, and contained in a fixed domain $\Omega \subset \Bbb{R}^3$.  To describe the smoothness assumption concerning $\Gamma(t)$ and its evolution more precisely, we introduce the Langrangian mapping from the space-time cylinder $\Gamma_0 \times [0,T]$, with $\Gamma_0:=\Gamma(0)$, to the space-time manifold
\[\Gs: = \bigcup\limits_{t \in (0,T)} \Gamma(t) \times \{t\},\quad  \Gs\subset \Bbb{R}^{4}, \]
see also \cite{ElliottStinner}. We assume that the velocity field $\bw$ and  $\Gamma_0$ are sufficiently smooth such that for all $y \in \Gamma_0$ the ODE system
\[
  \Phi(y,0)=y, \quad \frac{\partial \Phi}{\partial t}(y,t)= \bw(\Phi(y,t),t), \quad t\in [0,T],
\]
has a unique solution $x:=\Phi(y,t) \in \Gamma(t)$ (recall that $\Gamma(t)$ is transported with the  velocity field $\bw$). The corresponding inverse mapping is given by $\Phi^{-1}(x,t):=y \in \Gamma_0$, $x \in \Gamma(t)$. The Lagrangian mapping  $\Phi$  induces a bijection
\begin{equation} \label{defF}
 F:\,   \Gamma_0 \times [0,T] \to \Gs, \quad ~F(y,t):=(\Phi(y,t),t).
\end{equation}
\emph{We assume this bijection to be a $C^2$-diffeomorphism between these manifolds}. The conservation of a scalar quantity
$u: \Gs \to \Bbb{R}$ with a diffusive flux on $\Gamma(t)$ leads to the surface PDE (cf. \cite{James04}):
\begin{equation}
\dot{u} + ({\Div}_\Gamma\bw)u -{ \nu_d}\Delta_{\Gamma} u=0\qquad\text{on}~~\Gamma(t), ~~t\in (0,T],
\label{transport}
\end{equation}
 with initial condition $u(x,0)=u_0(x)$ for $x \in \Gamma(0)$. Here
\[ \dot{u}= \frac{\partial u}{\partial t} + \bw \cdot \nabla u
 \]
 denotes the advective material derivative,
 ${\Div}_\Gamma:=\operatorname{tr}\left( (I-\bn\bn^T)\nabla\right)$ is the surface divergence and $\Delta_\Gamma$ is the
 Laplace-Beltrami operator, $\nu_d>0$ is the constant diffusion coefficient.
Let $H^1(\Gs)$ be the usual Sobolev space on $\Gs$.
The following weak formulation of \eqref{transport} was shown to be well-posed in \cite{Dziuk07}:
Find $u \in H^1(\Gs)$ such that $u(\cdot,0) = u_0$ and for almost all $t \in (0,T]$
\begin{equation} \label{weakDziuk}
  \int_{\Gamma(t)} \dot{u} v + u v \, {\Div}_\Gamma \bw + \nu_d  \nablaG u \cdot \nablaG v  \, ds  = 0 \quad  \text{for all} ~ v(\cdot,t) \in H^1\big(\Gamma(t)\big).
\end{equation}
Here $\nablaG$ is the tangential gradient for $\Gamma(t)$. A similar weak formulation is considered in \cite{Vierling}.
The formulation~\eqref{weakDziuk} is a natural starting point for the \emph{Lagrangian} type finite element methods treated in  \cite{Dziuk07,Vierling}.
It is, however, less suitable for the  \emph{Eulerian} finite element method that we introduce in this paper.
Our discretization method uses the framework of space-time finite element methods. Therefore, it is natural  to consider
a space-time weak formulation of \eqref{transport} as given below.
We introduce the space
\[
H=\{\, v \in L^2(\Gs)~|~ \|\nablaG v\|_{L^2(\Gs)} <\infty \, \}
\]
endowed with the scalar product
\begin{equation}
(u,v)_H=(u,v)_{L^2(\Gs)}+ (\nablaG u, \nablaG v)_{L^2(\Gs)}, \label{inner}
  \end{equation}
and consider the material derivative as a linear functional on $H$. The subspace of all functions $v$ from $H$ such that
$\dot{v}$ defines a bounded linear functional form the trial space $W$. A precise definition of the space $W$ is given { in} Section~\ref{sectHW}. We consider the following weak formulation of  \eqref{transport}: Find $u\in W$ such that
\begin{equation} \label{weakSpaceTime}
\begin{split}
  \la\dot{u}, v\ra + \int_0^T \int_{\Gamma(t)} u v \, {\Div}_\Gamma \bw + {\nu_d} \nablaG u \cdot \nablaG v  \, ds dt& = 0\qquad \text{for all} ~ v \in H, \\
     u(\cdot,0) &= u_0.
\end{split}
\end{equation}
 We shall derive certain density properties for the
spaces $W$ and $H$, which we need for proving the well-posedness of \eqref{weakSpaceTime}. Actually, we
show  well-posedness of a slightly more general problem, which includes a possibly nonzero  source term
 and, instead of $ ({\Div}_\Gamma \bw) u$, a generic zero order term.  Our finite element method will be based on \eqref{weakSpaceTime} rather than \eqref{weakDziuk}.

\section{Preliminaries} \label{sec3}
In this section, we  define the trial space $W$ and prove that both the test space $H$ and the trial space $W$
are Hilbert spaces, and that smooth functions are dense in  $H$ and $W$. We also prove that
a function from $W$ has a well-defined trace as an element of $L^2(\Gamma(t))$ for all $[0,T]$. In the setting of a space-time manifold, the spaces $H$ and $W$ are  natural ones. In the literature, however, we did not find any analysis of their properties. The necessary results are established with the
help of a homeomorphisms between $H$, $W$ and the following standard  Bochner spaces $\widehat{H}$ and $\widehat{W}$:
\begin{equation}\label{WX}
\widehat{H}:=L^2(0,T;H^1(\Gamma_0)),\qquad \widehat{W}:= \{\, u \in \widehat{H}~|~\frac{\partial u}{\partial t} \in \widehat{H}'\,\},\qquad \Gamma_0:=\Gamma(0).
\end{equation}
In the next { subsection}, we  collect a few  properties of the Bochner spaces $\widehat{H}$ and $\widehat{W}$ that we need in our analysis.

\subsection{Properties of the spaces $\widehat{H}$ and $\widehat{W}$}
The  spaces $\widehat{H}$ and $\widehat{W}$ are endowed with the norms
\[
\|u\|_{\widehat{H}} =\left(\int_0^T \|u(t)\|_{H^1(\Gamma_0)}^2 \, dt\right)^{\frac12}\quad \text{and}\quad
\|u\|_{\widehat{W}}=\left(\|u\|_{\widehat{H}}^2+\left\|\frac{\partial u}{\partial t}\right\|_{\widehat{H}'}^2\right)^{\frac12}.
\]
We start with the following well-known result.
\begin{lemma} \label{lemma1}
 The space $C_0^\infty(\Gamma_0 \times (0,T))$ is dense in $\widehat{H}$.
\end{lemma}
\begin{proof}  The inclusion $C_0^\infty (\Gamma \times (0,T)) \subset \widehat{H}$ is trivial.
 By construction of the Bochner space, the set of simple functions
$
  \left\{\, \sum_{i=1}^n \chi_{B_i} \phi_i ~|~ \phi_i \in H^1(\Gamma_0),~~n \in \Bbb{N} \, \right\}
$
is dense in $\widehat{H}$; here  $\{B_i\}$ is any set of $n$ mutually disjoint measurable subsets of $(0,T)$.
For $\chi_{B_i} \phi_i$ there exists $g_i \in C_0^\infty(0,T)$ and $\widehat \phi_i \in C^\infty(\Gamma_0)$
such that $\|\chi_{B_i} \phi_i -g_i\widehat \phi_i\|_{\widehat{H}}$ is arbitrary small. This completes the proof.
\end{proof}
\\[1ex]
For $w \in \widehat{H}$ we define the weak time derivative through the functional
\begin{equation} \label{r1}
   \la \frac{\partial w}{\partial t}, \phi \ra := -\int_0^T \int_{\Gamma_0} w(t) \frac{\partial \phi}{\partial t} \, ds \, dt \quad \text{for}~~\phi \in C^1_0(\Gamma_0 \times (0,T)).
\end{equation}
Then $ \frac{\partial w}{\partial t} \in \widehat{H}'$ iff there is a constant $c$ such that
\[
  \left| \la \frac{\partial w}{\partial t}, \phi \ra \right| \leq c \|\phi\|_{\widehat{H}} \quad \text{for all} ~~\phi \in C^1_0(\Gamma_0 \times (0,T)).
\]

\begin{remark}
\rm
The definition of $\widehat{W}$ in \eqref{WX}, based on the weak time derivative \eqref{r1}, is equivalent to the following more standard one: $w \in \widehat{H}$ is an element of $\widehat{W}$ iff there exists $z \in L^2(0,T;H^{-1}(\Gamma_0))$ such that
\begin{equation} \label{r2}
 \int_0^T \la z(t), v\ra_{H^{-1}\times H^1} \phi(t) \, dt = - \int_0^T \int_{\Gamma_0} w(t) v \phi'(t) \, ds \, dt
\end{equation}
for all $v \in H^{1}(\Gamma_0),~\phi \in C_0^\infty(0,T)$.
The definition of $\widehat{W}$ in \eqref{WX} is more convenient for the analysis that follows.
\end{remark}

We recall a few results for the space  $\widehat{W}$.

\begin{lemma} \label{lem4}
The set
\[
 \mathcal{C}=\left\{ \, \sum_{i=1}^n t^i \phi_i~|~\phi_i \in C^\infty(\Gamma_0), ~~n \in \Bbb{N}\,\right\}
\]
is dense in $\widehat{W}$. For $u \in \widehat{W}$ the function $t \to u(t)=u(\cdot,t)$ is continuous from $[0,T]$ into $L^2(\Gamma_0)$. There is  a constant $c$ such that
\begin{equation} \label{trace1}
 \max_{0 \leq t \leq T} \|u(t)\|_{L^2(\Gamma_0)} \leq c \|u\|_{\widehat{W}} \quad \text{for all} ~~u \in \widehat{W}.
\end{equation}
\end{lemma}
\begin{proof} Proofs are given in standard textbooks, e.g., \cite{Zeidler} Proposition 23.23. The density result is usually proved with
$ C^\infty(\Gamma_0)$ replaced by $H^{1}(\Gamma_0)$ in the definition of $\mathcal{C}$. The result with $ C^\infty(\Gamma_0)$ follows from the density of $ C^\infty(\Gamma_0)$ in $H^{1}(\Gamma_0)$.
\end{proof}

\subsection{The spaces $H$ and $W$} \label{sectHW}

We assume that the space-time surface $\Gs$ is sufficiently smooth, cf. Section~\ref{sec2}. 
Due to the identity
\begin{equation}\label{transform}
 \int_0^T \int_{\Gamma(t)} f(s,t) \, ds \, dt = \int_{\Gs} f(\sigma) (1+ (\bw \cdot \bn)^2)^{-\frac12}\, d\sigma,
\end{equation}
the scalar product $\int_0^T \int_{\Gamma(t)} v w \, ds \, dt$ induces a norm   that is equivalent to the standard norm on $L^2(\Gs)$. Therefore, one can equivalently define the norm on $H$ by
\begin{equation} \label{defH}
\|v\|_H^2:=\int_0^T \int_{\Gamma(t)} v^2+ |\nablaG v|^2ds \, dt.
\end{equation}
The space $H$ is a Hilbert space, and  $H \hookrightarrow L^2(\Gs) \hookrightarrow H'$ forms a Gelfand triple (cf. Lemma~\ref{PropH}
below).

Recall the Leibniz formula
\begin{equation}\label{Leibniz}
  \int_{\Gamma(t)} \dot v + v \DivG \bw \, ds = \frac{d}{dt}  \int_{\Gamma(t)} v \, ds, \quad v\in C^1(\Gs),
\end{equation}
which implies the  integration by parts  identity:
\begin{equation}\label{ByParts}
 \begin{split}
  & \int_0^T \int_{\Gamma(T)} \dot uv +\dot vu + u v \DivG \bw \, ds\, dt\\  & = \int_{\Gamma(T)} u(s,T) v(s,T)\, ds - \int_{\Gamma(0)} u(s,0) v(s,0)\, ds \quad \text{for all}~u,v\in C^1(\Gs).
\end{split}
\end{equation}


Based on \eqref{ByParts} we  define  the material derivative  for $u \in H$ as the functional $\dot u$:
\begin{equation} \label{weakmaterial}
 \la\dot u,\phi\ra= - \int_0^T \int_{\Gamma(t)} u \dot \phi + u \phi \DivG \bw\, ds \, dt  \quad \text{for all}~~ \phi \in C_0^1(\Gs).
\end{equation}
Assume that for some $u\in H$ the norm
\[
\|\dot{u}\|_{H'}=\sup_{\phi\in C_0^1(\Gs)}\frac{\la\dot u,\phi\ra}{\|\phi\|_{H}}
\]
is bounded. In~Lemma~\ref{PropH} we prove that $C_0^1(\Gs)$ is  dense in $H$ and therefore $\dot{u}$ can be extended to a bounded linear functional on $H$. In this case,
we write $\dot u \in H'$ and  define the space
\[
  W= \{ \, u\in H~|~\dot u \in H' \,\}, \quad \text{with}~~\|u\|_W^2 := \|u\|_H^2 +\|\dot u\|_{H'}^2.
\]
In Section~\ref{secttHW}  we prove that $W$ is a Hilbert space and $C^1(\Gs)$ is  dense in $W$.
Note that the space $W$ is larger than the standard Sobolev space $H^1(\Gs)$. Spaces similar to $H$ and $W$ are introduced and analyzed in \cite{Vierling}. A difference between our aproach and the one  used in that paper is that we define $H$ and $W$  directly on the space-time manifold $\Gs$, whereas in \cite{Vierling} these are defined using a pull back operator to the manifold $\Gamma_0 \times [0,T]$. We use such a pull back operator in the analysis of the spaces $H$ and $W$ in the next section, but not in their definition.
\begin{remark}
 \rm From the definition of the weak material derivative in \eqref{weakmaterial} and the density result of Lemma~\ref{PropH} it follows that for $u \in C^1(\Gs)$ we have
\begin{equation*} \label{id1}
 \langle \dot u , v\rangle = \int_0^T \int_{\Gamma(t)} \dot u v  \, ds \, dt \quad \text{for all}~~v \in H.
\end{equation*}
\end{remark}

\subsection{ Homeomorphism between \{$\widehat{H}$, $\widehat{W}$\} and \{$H$, $W$\}}
\label{secHomo}
Based on the  $C^2$-diffeomorphism $F$ in  \eqref{defF},
for a function $u$ defined on $\Gs$  we define $\widehat u=u\circ F$ on $\Gamma_0 \times (0,T)$:
\[
 \widehat u(y,t)=u(\Phi(y,t),t)=u(x,t).
\]
Vice versa, for a function $\widehat u$ defined on $\Gamma_0 \times (0,T)$  we define $u=\widehat u\circ F^{-1}$ on $\Gs$:
\begin{equation}\label{homo}
u(x,t)=\widehat u(\Phi^{-1}(x,t),t)=\widehat u(y,t).
\end{equation}
By construction we have
\begin{equation} \label{transf1}
  \dot u (x,t)= \frac{\partial \widehat u}{\partial t}(y,t).
\end{equation}
Now we prove that the mapping $\widehat u \to u$ defines a linear homeomorphism between  $\widehat{H}$ and $H$, and also between
$\widehat{W}$ and $W$.

\begin{lemma} \label{homoH}
The linear mapping $\widehat u \to u$ from \eqref{homo} defines a  homeomorphism between $\widehat{H}$ and $H$.
\end{lemma}
\begin{proof}
 For any fixed $t\in(0,T)$, we obtain $\widehat u(t):=\widehat u(\cdot,t) \in H^1(\Gamma_0)$ iff $u(t)=u(\cdot,t) \in H^1(\Gamma(t))$.  Let $\widehat u(t) \in H^1(\Gamma_0)$ for  all $t \in (0,T)$. Due to the smoothness assumptions on $F$,  there are constants $c_1, c_0 >0$, independent of $\widehat u$ and $t$,  such that
\begin{equation}
 c_0 \|\widehat u(t)\|_{H^1(\Gamma_0)} \leq \|u(t)\|_{H^1(\Gamma(t))} \leq c_1 \|\widehat u(t)\|_{H^1(\Gamma_0)} \quad \text{for all}~~t \in (0,T).
\end{equation}
Hence, $\widehat u \in \widehat{H}$ iff $u \in H$ holds, and the linear mapping $\widehat u \to u$ is a homeomorphism between $\widehat{H}$ and $H$.
\end{proof}
\\[1ex]
For the further analysis, we need a surface integral transformation formula. For this we consider  a local parametrization of $\Gamma_0$, denoted by $\mu:\Bbb{R}^2 \to \Gamma_0$, which is at least $C^1$ smooth. Then, $\Phi \circ \mu:=\Phi(\mu(\cdot), t)$ defines a $C^1$ smooth parametrization of $\Gamma(t)$.    For the surface measures $d\,\widehat{s}$ and $ds$ on $\Gamma_0$ and $\Gamma(t)$, respectively, we have the relation
\begin{equation} \label{fromtrans1}
  ds = \frac{| \nabla_\Gamma \Phi(\cdot ,t) \mu_x \times \nabla_\Gamma \Phi(\cdot,t) \mu_y|}{|\mu_x \times \mu_y|}\, d\,\widehat{s} =:\gamma(\cdot ,t)\, d\,\widehat{s},
\end{equation}
with $\mu_x=\frac{\partial \mu}{\partial x} \big(\mu^{-1}(\cdot)\big)$, and similarly for $\mu_y$. Recall that $ \nabla_\Gamma f$ denotes the $\Gamma(t)$-surface gradient of a scalar function $f$ defined on $\Gamma(t)$ for any fixed $t$.
Using this integral transformation formula, for  $u \in H$ and $\phi \in  C_0^1(\Gs)$ we obtain
\begin{equation} \label{transf2}
\la\dot u,\phi\ra= - \int_0^T \int_{\Gamma(t)} u \dot \phi \,  + u \phi \DivG \bw\, ds \, dt = -  \int_0^T \int_{\Gamma_0} (\widehat u \frac{\partial \widehat \phi}{\partial t}  +\widehat u \widehat \phi \widehat{\DivG \bw}\, )\gamma \, d\,\widehat{s} \, dt
\end{equation}

\begin{lemma} \label{homoW}  The linear mapping $\widehat u \to u$ from \eqref{homo} defines a  homeomorphism between $\widehat W$ and $W$.
\end{lemma}
\begin{proof}
 The proof makes use of the formula \eqref{transf2}.
Take $u \in H$ with $\widehat u \in \widehat{W}$, and $\phi \in C_0^1(\Gs)$. We use the notation  $\lesssim$ if the constant that occurs in the inequality  is independent of $u$ and $\phi$, and $\simeq$ if such an inequality holds in two directions.
Due to the $C^2$-smoothness assumption on $F$ (and thus $\Phi$) the function $\gamma$ defined in \eqref{fromtrans1} is $C^1$-smooth on $\Gamma_0 \times [0,T]$. Define $\widehat \psi:= \widehat \phi \gamma \in C_0^1(\Gamma_0 \times (0,T))$. Due to Lemma~\ref{homoH}
we have $\|\widehat \psi\|_{\widehat{H}} \lesssim \|\widehat \phi \|_{\widehat{H}} \simeq \|\phi\|_H$.
Therefore, we can estimate
\begin{equation*}
 \begin{split}
 |\la \dot u, \phi \ra | & = \big |\int_0^T \int_{\Gamma_0} (\widehat u \frac{\partial \widehat \phi}{\partial t} +\widehat u \widehat \phi \widehat{\DivG \bw}\, )\gamma  \, d\,\widehat{s} \, dt \big|
  \\ &\leq \big |\int_0^T \int_{\Gamma_0} \widehat u \frac{\partial \widehat \psi}{\partial t}\, d \widehat{s} \, dt \big| + \big| \int_0^T \int_{\Gamma_0} \widehat u \widehat \phi \frac{\partial \gamma}{\partial t} \, d \widehat{s} \, dt \big|
  + \big| \int_0^T \int_{\Gamma_0} \widehat u \widehat \phi { \widehat{\DivG \bw}}\, \gamma\, d \widehat{s} \, dt \big|
\\ & \lesssim \|\widehat u \|_{\widehat{W}} \|\widehat \psi\|_{\widehat{H}} + \|\widehat u \|_{\widehat{H}} \|\widehat \phi\|_{\widehat{H}} \lesssim \|\widehat u \|_{\widehat{W}} \|\phi\|_H.
\end{split} \end{equation*}
Hence, $u \in W$ and $\|u\|_W \lesssim \|\widehat u\|_{\widehat{W}}$ holds. With similar arguments one can show that if $u \in W$, then $\widehat u\in \widehat{W}$ and
$ \|\widehat u\|_{\widehat{W}} \lesssim \|u\|_W$ holds. For this, instead of the surface transformation formula
\eqref{fromtrans1}  one starts with the formula
\begin{equation}\label{wg}
 d\,\widehat{s}=
 \frac{|\mu_x \times \mu_y|}{| \nabla_\Gamma \Phi(\Phi^{-1}(\cdot,t),t) \mu_x \times \nabla_\Gamma \Phi(\Phi^{-1}(\cdot,t),t) \mu_y|}\, ds = :\widetilde\gamma(\cdot,t) \, ds,
\end{equation}
with $\mu_x=\mu_x\big((\Phi\circ \mu)^{-1}(\cdot)\big)$, and similarly for $\mu_y$.
For $\widehat u \in \widehat{H}$ and $\widehat \phi \in C_0^1(\Gamma_0 \times (0,T))$ we have
\[ \la\frac{\partial \widehat u}{\partial t}, \widehat \phi\ra= -  \int_0^T \int_{\Gamma_0} \widehat u \frac{\partial \widehat \phi}{\partial t} \, d\,\widehat s \, dt=
- \int_0^T \int_{\Gamma(t)} u \dot \phi  \widetilde \gamma \, ds \, dt
\]
Now we note that $\widetilde\gamma$ is $C^1$-smooth on $\Gs$. To check this, due to the $C^2$-diffeomorphism property of $\Phi$ it is sufficient to show that the denominator in \eqref{wg} is uniformly bounded away from zero on $\Gs$.
For ${ \tilde x} \in \Gamma(t)$ and with  ${\tilde x}=(\Phi \circ \mu)(z)$ one can rewrite the denominator as
\begin{equation}\label{aux1} \begin{split}
  | \nabla_\Gamma \Phi(\Phi^{-1}(\tilde x,t),t) \mu_x(z) \times \nabla_\Gamma \Phi(\Phi^{-1}(\tilde x,t),t) \mu_y(z)| \\  =  | \nabla_\Gamma \Phi(\mu(z),t) \mu_x(z) \times \nabla_\Gamma \Phi(\mu(z),t) \mu_y(z)|   = |\frac{\partial }{\partial x}(\Phi \circ \mu)(z) \times\frac{\partial }{\partial y}(\Phi \circ \mu)(z)|.
\end{split}
\end{equation}
 From the fact that $\Phi\circ \mu$ is a $C^1$-smooth parametrization of $\Gamma(t)$ it follows that the quantity on the right-hand side of \eqref{aux1} is uniformly bounded away from zero.
Hence, the function $\widetilde\gamma$ is $C^1$-smooth and we can use the same arguments as above to derive $ \|\widehat u\|_{\widehat{W}} \lesssim \|u\|_W$.
This implies that $\widehat u \to u$ is a  homeomorphism between $\widehat{W}$ and $W$.
 \end{proof}

\subsection{Properties of $H$ and $W$} \label{secttHW}

The homeomorphism  established in Section \ref{secHomo} helps us to  derive  density results for the spaces $H$ and $W$  and a trace property similar to the one
in \eqref{trace1}.

\begin{lemma} \label{PropH} $H$ is a Hilbert space.
The space $C_0^1(\Gs)$ is dense in $H$. The spaces $\{H,L^2(\Gs),H'\}$ form a Gelfand triple.
\end{lemma}
\begin{proof} Let $L\,:\, \widehat{H}\to H$ denote the mapping given in \eqref{homo}.
Since $L$ is a linear homeomorphism between, the space $H$ is complete and so this is a Hilbert space.
For $\widehat \phi \in C_0^1(\Gamma_0 \times (0,T)) \subset \widehat{H}$ we have, due to the smoothness assumptions on $F$, that $\phi = L \widehat \phi \in C^1(\Gs)$. Furthermore, from $\phi(x,t)=\widehat \phi(\Phi^{-1}(x,t),t)$ it follows that $\phi$ has compact support. Hence, $\phi \in C_0^1(\Gs)$. From this we get $L\big[ C_0^1(\Gamma_0 \times (0,T))\big] \subset C_0^1(\Gs)$. Since $ C_0^1(\Gamma_0 \times (0,T))$ is dense in $\widehat{H}$ and $L$ is a homeomorphism, this implies that $C_0^1(\Gs)$ is dense in $H$. Since $C_0^1(\Gs)$ is also dense in $L^2(\Gs)$, the space $H$ is dense in $L^2(\Gs)$.
Hence, $\{H,L^2(\Gs),H'\}$ is a Gelfand triple.
\end{proof}
\medskip

For $t \in [0,T]$ and $u \in C^1(\Gs)$ denote by $u \to u_{|\Gamma(t)}$ a trace operator:  $u_{|\Gamma(t)}(x)=u(x,t)$, $x \in \Gamma(t)$.
In Section~\ref{ssectDG} we analyze a discontinuous Galerkin method in time. For such a method, one needs well defined traces of this type.
For a smooth function $\widehat{u}(x,t)$ defined on the cylinder $\Gamma_0\times[0,T]$, it is obvious that
one can define at any time $t\in[0,T)$, the right limit $\widehat{u}_{+}(\cdot,t)=\lim\limits_{\delta\to+0}{ \widehat{u}}(\cdot, t+\delta)$ on $\Gamma_0$. Similarly a left limit function $\widehat{u}_{-}(\cdot,t)$ is defined for  $t\in(0,T]$.
For a sufficiently smooth function $u$  on $\Gs$, due to the fact that the domain $\Gamma(t)$ where the trace has to be defined varies with $t$,  it is less straightforward to  construct such left and right limit functions.  To this end, for $u \in C^1(\Gs)$ and a given $t \in [0,T]$ we define $u_\delta: \, \Gamma(t) \to \Bbb{R}$ by
\begin{equation} \label{conti}
   u_\delta(\cdot,t):= u\big(F(\Phi^{-1}(\cdot, t),t+\delta)\big),~~\delta~\text{such that}~~t+ \delta \in (0,T).
\end{equation}
Note that { $u_{\delta}=u_{|\Gamma(t)}$ holds when $\delta=0$}.
Right and left limits on $\Gamma(t)$ are defined as
\begin{equation} \label{conti2}
{u}_{+}(\cdot,t)  =\lim\limits_{\delta\to+0}u_\delta(\cdot,t)\quad\text{for}~t\in [0,T), \qquad
{u}_{-}(\cdot,t) =\lim_{\delta\to-0}u_\delta(\cdot,t)\quad\text{for}~t\in (0,T].
\end{equation}
Below we show that for functions from $W$ the trace and one-sided limits are well-defined and can be considered as  elements of $L^2(\Gamma(t))$.\smallskip

The next theorem gives several important properties for our trial space.\smallskip

\begin{theorem} \label{PropW}  $W$ is a Hilbert space and has the following properties:
\begin{description}
\item(i)~$C^1(\Gs)$ is dense in $W$.
\item(ii)~For every $t \in [0,T]$ the trace operator $u \to u_{|\Gamma(t)}$ can be extended to a bounded linear operator from $W$ to $L^2(\Gamma(t))$.  Moreover, the inequality
\begin{equation} \label{rrt}
 \max_{0 \leq t\leq T} \|u_{|\Gamma(t)}\|_{L^2(\Gamma(t))} \leq c \|u\|_W \quad \text{for all}~~ u \in W,
\end{equation}
holds with a constant $c$ independent of $u$.
\item(iii)~Take $t \in [0,T)$ and let $\ep >0$ be sufficiently small, such that $t+\ep\le T$.   For any $u \in W$ the mapping
$\delta \to u_\delta(\cdot,t)$ defined in \eqref{conti} is continuous from $[0,\ep]$ into $L^2(\Gamma(t))$. The same assertion is true for $t \in (0,T]$ and suitable $\ep<0$. For $u \in W$ the one-sided limits \eqref{conti2} are well-defined.
\end{description}
\end{theorem}
\begin{proof}
Since the mapping $L: \widehat u \to u$ given by \eqref{homo} is a linear homeomorphism between $\widehat{W}$ and $W$, the space $W$ is complete and so this is a Hilbert space.

(i)~Let $\mathcal{C}$ be the set as in Lemma~\ref{lem4}, which is dense in $\widehat{W}$. One easily checks $L(\mathcal{C}) \subset C^1(\Gs)$. Since $L(\mathcal{C})$ is dense in $W$, this implies that $C^1(\Gs)$ is dense in $W$.

(ii)~Take $u \in C^1(\Gs)$ and define $u_{|\Gamma(t)}:=u(\cdot, t) \in L^2(\Gamma(t))$.  Using \eqref{fromtrans1}, Lemma~\ref{lem4} and Lemma~\ref{homoW} we get
\[
  \|u_{|\Gamma(t)}\|_{L^2(\Gamma(t))} \leq c \|\widehat u(t)\|_{L^2(\Gamma_0)} \leq c \|\widehat u\|_{\widehat{W}} \leq c \|u\|_W,
\]
where the constant $c$ can be assumed to be independent of $t$ due to the smoothness of $\gamma$ in \eqref{fromtrans1}. From this, the result in \eqref{rrt} follows by a density argument.

(iii)~Take a fixed $t \in [0,T)$ and sufficiently small $\ep>0$. Take $\delta_1, \delta_2 \in [0,\ep]$. For $x \in \Gamma(t)$ we use the substitution  $y:=\Phi^{-1}(x,t) \in \Gamma_0$ and the integral transformation formula as in the proof of Lemma~\ref{homoH}, resulting in:
\begin{align*}
\|u_{\delta_1}(\cdot,t)-u_{\delta_2}(\cdot,t)\|_{L^2(\Gamma(t))}
&  = \|u \big(F(\Phi^{-1}(\cdot,t),t+\delta_1)\big) - u \big(F(\Phi^{-1}(\cdot,t),t+\delta_2)\big)\|_{L^2(\Gamma(t))} \\
 &  \leq c \|\widehat u (\cdot ,t+\delta_1)-\widehat u(\cdot,t+\delta_2)\|_{L^2(\Gamma_0)},
\end{align*}
with a constant $c$ independent of $t$. Hence, the continuity of the mapping $\delta \to u_\delta(\cdot,t)$  follows from the continuity result for $\widehat u$ in Lemma~\ref{lem4}. Due to the continuity of the mappings, the one-sided limits are well-defined.
 \end{proof}

\begin{corollary} \label{Partialint}
 For all $u,v \in W$, the integration by parts  identity holds:
\begin{equation} \label{partint}
\begin{split}
 \la\dot u,v\ra &+\la\dot v,u\ra + \int_0^T \int_{\Gamma(T)}u v \DivG \bw \, ds\, dt\\  &  = \int_{\Gamma(T)} u(s,T) v(s,T)\, ds - \int_{\Gamma(0)} u(s,0) v(s,0)\, ds.
\end{split}
\end{equation}
\end{corollary}
\begin{proof}
Follows from the identity \eqref{ByParts} and Theorem~\ref{PropW}.
\end{proof}

\section{Well-posedness of  weak formulation} \label{sec4}

Using the properties of $H$ and $W$ derived above, we prove a well-posedness result for the weak space-time formulation \eqref{weakSpaceTime} of the surface transport-diffusion equation \eqref{transport}. The analysis uses the LBB approach and is along the same lines as presented for a parabolic problem on a fixed Euclidean domain in \cite{Ern04} (Section 6.1).
As usual, we first transform the problem \eqref{transport} to ensure that the initial condition
is homogeneous. To this end, consider the decomposition of the solution $u=\widetilde{u}+u^0$, where $u^0:\ \Gs \to \Bbb{R}$ is chosen sufficiently smooth and such that
\[
u^0(x,0)=u_0(x)\quad\text{{on $\Gamma_0$,} and}\quad\frac{d}{dt}\int_{\Gamma(t)} u^0 \, ds =0.
\]
One can set,  e.g., $u^0=(u_0\circ\Phi^{-1})(\gamma\circ F^{-1})^{-1}$,
with $\gamma$ from \eqref{fromtrans1}. Since the solution of \eqref{transport} has the  mass conservation property $\frac{d}{dt}\int_{\Gamma(t)} u\,ds=0$, and $\int_{\Gamma(0)} u^0\,ds= \int_{\Gamma(0)} u\,ds$ by the choice of $u^0$, the new unknown function $\widetilde{u}$ satisfies $\tilde u(\cdot,0)=0$ on $\Gamma_0$ and
\begin{equation}\label{average}
\int_{\Gamma(t)} \widetilde{u}\,ds=0\quad\text{for all}~t\in [0,T].
\end{equation}
For this transformed function the diffusion equation takes the form
\begin{equation}
\begin{aligned}
\dot{\widetilde{u}} + ({\Div}_\Gamma\bw)\widetilde{u} -{\nu_d}\Delta_{\Gamma} \widetilde{u}&=f\qquad\text{on}~~\Gamma(t), ~~t\in (0,T],\\
\widetilde{u}(\cdot,0)&=0\qquad\text{on}~~\Gamma_0.
\end{aligned}
\label{transport1}
\end{equation}
The right-hand side is now non-zero: $f:=-\dot{u^0} - ({\Div}_\Gamma\bw)u^0 + {\nu_d}\Delta_{\Gamma}u^0 \in H'$.
Using \eqref{partint}  and the integration by parts over $\Gamma(t)$, one immediately finds
$\la f,\phi\ra=0$ for $\phi\in C_0^1(0,T)$. For a more regular source function, $f\in L^2(\Gs)$, this implies $\int_{\Gamma(t)} f\,ds=0$ for almost  all $t\in [0,T]$.

In the analysis below, instead of the (transformed) diffusion problem \eqref{transport1} we consider the following
 slightly more general surface PDE:
\begin{equation}
\begin{aligned}
\dot{u} + \alpha\, {u} -{\nu_d}\Delta_{\Gamma} {u}&=f\qquad\text{on}~~\Gamma(t), ~~t\in (0,T],\\
{u}&=0\qquad\text{on}~~\Gamma_0,
\end{aligned}
\label{transport2}
\end{equation}
with $\alpha\in L^\infty(\Gs)$ and a generic right-hand side $f\in H'$, not necessarily satisfying the zero integral condition. We use the notation $\alpha_{\infty}:=\|\alpha\|_{L^\infty(\Gs)}$.
\medskip

We define the inner product and symmetric bilinear form
\[
(u,v)_0=\int_0^T \int_{\Gamma(t)} u v \, ds \, dt,\quad
  a(u,v)= \nu_d (\nablaG u, \nablaG v)_0 + (\alpha u,v)_0, \quad u, v \in H.
\]
This bilinear form is  continuous on $H\times H$:
\begin{equation}\label{eq:continuity}
a(u,v)\le(\nu_d+\alpha_{\infty}) \|u\|_H\|v\|_H.
\end{equation}
Consider the subspace of $W$ of all function vanishing for $t=0$:
\[
\Wo:=\{\, v \in W~|~v(\cdot, 0)=0 \quad \text{on}~\Gamma_0\,\}.
\]
The space $\Wo$ is well-defined,
since functions from $W$ have well-defined traces on $\Gamma(t)$ for any $t\in[0,T]$, see Theorem~\ref{PropW}.
The weak space-time formulation of \eqref{transport2} reads: Given $f\in H'$, find $u \in \Wo$ such that
\begin{equation} \label{weakformu}
 \la\dot u,v\ra +a (u,v) =\la f,v\ra \quad \text{for all}~~v \in H.
\end{equation}

In the remainder of this section we prove that this variational problem  is well-posed.
Our analysis is based on the continuity and inf-sup conditions, cf.~\cite{Ern04}. The continuity property is straightforward:
\[
  | \la\dot u,v\ra + a(u,v)| \le (1+\nu_d+\alpha_{\infty})\|u\|_W \|v\|_H \quad \text{for all}~~u \in W,~v \in H.
\]

The next two lemmas are crucial for proving the well-posedness of  \eqref{weakformu}.

\begin{lemma}\label{la:infsup}
The inf-sup inequality
 \begin{equation}\label{infsup}
  \inf_{0\neq u \in \Wo}~\sup_{ 0\neq v \in \overset{\phantom{.}}{H}} \frac{\la\dot u,v\ra + a(u,v)}{\|u\|_W\|v\|_H} \geq c_s
 \end{equation}
holds with some $c_s>0$.
\end{lemma}
\begin{proof}
Take $u\in \Wo$. In \eqref{weakformu} we take a test function $v=u_\ga:= e^{-\gamma t}u \in \Wo$, with $\gamma:=2(\nu_d+\|\alpha-\frac12\DivG \bw\|_{L^\infty(\Gs)})$. We note the identity:
\begin{equation} \label{hp}
 \la \dot u_\gamma,u\ra= \la \dot u, u_\gamma \ra - \gamma(u, u_\gamma)_0.
\end{equation}
 From \eqref{hp}, \eqref{partint}, and condition $u(0)=0$, we infer
\[
  \la \dot u, u_\gamma \ra = \frac12\big(\la \dot u, u_\gamma \ra +  \la \dot u_\gamma,u\ra \big) + \frac12 \gamma(u, u_\gamma)_0
\geq -\frac12 (u, u_\gamma  \DivG \bw)_0 + \frac12 \gamma(u, u_\gamma)_0.
\]
This and the choice of $\gamma$ implies
$
  \la \dot u, u_\gamma \ra+ (\alpha u, u_\gamma)_0 \geq \nu_d(u, u_\gamma)_0 \geq \nu_d e^{-\gamma T} \|u\|_0^2.
$
Combining this with $(\nablaG u,\nablaG u_\gamma)_0 \geq  e^{-\gamma T} \|\nabla_\Gamma u\|_0^2$, we get
\begin{equation}\label{eq:proofinfsup1}
  \la\dot{u},u_\ga\ra +  a(u,u_\ga) \ge \nu_d e^{-\ga T}\|u\|_H^2.
\end{equation}
This establishes the control of $\|u\|_H$ on the right-hand side of the inf-sup inequality. We also need control of $\|\dot u\|_{H'}$ to bound the full norm $\|u\|_W$.  This is achieved  by using a duality argument between the Hilbert spaces $H$ and $H'$.

By Riesz' representation theorem, there is a unique $z\in H$ such that $\la\dot u,v\ra = (z, v)_H$ for all $v\in H$, and $\|z\|_H = \|\dot{u}\|_{H'}$ holds. Thus we obtain
\[
  \la\dot{u},z\ra = (z,z)_H = \|\dot{u}\|_{H'}^2.
\]
Therefore, with the help of~\eqref{eq:continuity}, we get
\begin{equation}\label{eq:proofinfsup2}
 \begin{split}
  \la \dot{u},z\ra +  a(u, z) & = \|z\|_H^2 +  a(u,z) \ge \|z\|_H^2 - \frac{c_1}2  \|u\|_H^2 - \frac12\|z\|_H^2 \\
   & = \frac12\|\dot{u}\|_{H'}^2 - \frac{c_1}2  \|u\|_H^2,\qquad \text{with}~c_1=(\nu_d+\alpha_{\infty})^2.
\end{split}
\end{equation}
This establishes control of $\|\dot{u}\|_{H'}$ at the expense of the $H$-norm, which is controlled in \eqref{eq:proofinfsup1}.
Therefore, we make the ansatz $v= z + \mu u_\ga\in H$ for some sufficiently large parameter $\mu \geq 1$. We have the estimate
\begin{equation}\label{aux20}
\|v\|_H\le \|z\|_H + \mu\|u_\ga\|_H \le \|\dot u\|_{H'} + \mu\|u\|_H\le \mu \sqrt{2} \|u\|_W.
\end{equation}
 From \eqref{eq:proofinfsup1}, \eqref{eq:proofinfsup2}, and \eqref{aux20} we conclude
\[
  \la\dot u,v\ra +  a(u,v) \ge \frac12\|\dot{u}\|_{H'}^2 +(\mu \nu_d e^{-\ga T} - \frac{c_1}2)\|u\|_H^2.
\]
Taking $\mu:=\frac1{2\nu_d}(c_1+1)e^{\gamma T}$, we get
\[
   \la\dot u,v\ra +  a(u,v) \ge \frac12\|u\|_W^2 \ge \frac{\sqrt{2}}{4} \mu^{-1}\|u\|_W\|v\|_H.
\]
This completes the proof.
\end{proof}

\begin{remark} \label{remstab} \rm
A closer look at the proof reveals that the stability constant $c_s$ in the inf-sup condition \eqref{infsup} can be taken
as
\[
c_s= \frac{\nu_d}{\sqrt{2}}(1+\nu_d+\alpha_{\infty})^{-2} e^{-2T(\nu_d+\tilde c)},\quad \tilde c =\|\alpha-\frac12\DivG \bw\|_{L^\infty(\Gs)}.
\]
This stability constant deteriorates if $\nu_d \downarrow 0$ or $T \to \infty$. We do not consider the singularly perturbed case of vanishing diffusion. 
Without assumptions on the sign of $\alpha$ and the size of the velocity field $\bw$ (which is part of the material derivative), there may be (exponentially) growing components in the solution and thus the exponential decrease of the stability constant as a function of $T$ can not be avoided.  In special cases, the behavior of the stability constant may be better, e.g. bounded away from zero uniformly in $T$. We comment on this further after Theorem~\ref{mainthm1}.
\end{remark}

\begin{lemma} \label{la:leminj}
If $\la\dot u,v\ra + a(u,v)=0$  for some $v\in H$ and all $u \in \Wo$, then $v=0$.
\end{lemma}
\begin{proof}
Take $v \in H$ such that $\la\dot u,v\ra =- a(u,v) $ for all $u \in \Wo$. For all $ u \in C_0^1(\Gs) \subset \Wo$ we have by definition
\begin{align*}
   \la\dot v,u\ra & = - \int_0^T \int_{\Gamma(t)} v \dot u  \, ds \, dt- \int_0^T \int_{\Gamma(t)}u v \DivG \bw \, ds\, dt\\ &
    = -\la\dot u,v\ra - \int_0^T \int_{\Gamma(t)}u v \DivG \bw \, ds\, dt   =  a(u,v) -(u,v\DivG \bw)_0.
\end{align*}
Since the functional  $u \to a(u,v)-(u,v\DivG \bw)_0 $ is in $H'$,  we conclude $\dot v \in H'$, and thus $v \in W$ holds. From  a density argument it follows that
\begin{equation} \label{aux2}
\la\dot v,u\ra=  a(u,v) -(u,v\DivG \bw)_0  \quad \text{for all}~~u \in H
\end{equation}
holds.
 For all $u \in \Wo \subset H$ we get
 \[ \la \dot v,u\ra = a(u,v) -(u,v\DivG \bw)_0 =-\la\dot u,v\ra -(u,v\DivG \bw)_0.\]
This and \eqref{partint} yield
\[
  0=  \la\dot v,u\ra +\la\dot u,v\ra+(u,v\DivG \bw)_0= \int_{\Gamma(T)} u(s,T) v(s,T)\, ds \quad \text{for all}~~u \in \Wo.
\]
This implies that $v(\cdot,T)=0$ on $\Gamma(T)$. We proceed as in the first step of the proof of Lemma \ref{la:infsup}. We take in \eqref{aux2} $u=v_{\ga}=e^{-\gamma t} v$, with $\gamma:=-2(1+\|\alpha-\frac12\DivG \bw\|_{L^\infty(\Gs)}) \le 0$, and use \eqref{hp}  and \eqref{partint}. We obtain
\begin{align*}
   0 & =\la\dot v,v_\ga\ra- a(v_\ga,v) +(v_\ga,v \DivG \bw)_0  \\
  & = \frac12( \la \dot v, v_\ga \ra + \la \dot v_\ga, v\ra ) + \frac12 \ga (v, v_\ga)_0 - a(v_\ga,v) +(v_\ga,v \DivG \bw)_0  \\
 & \leq \frac12 \ga (v, v_\ga)_0  + \frac12(v_\ga,v \DivG \bw)_0  -  a(v_\ga,v) \\
 & \leq   - (v, v_\ga)_0  - \nu_d(\nabla_\Gamma v_\ga,\nabla_\Gamma v)  \leq - (\|v\|_0^2 + \nu_d\|\nabla_\Gamma v\|_0^2).
\end{align*}
We conclude $v=0$.
\end{proof}
\medskip

As a direct consequence of the preceding two lemmas we obtain the following well-posedness result.

\begin{theorem} \label{mainthm1}
For any $f\in H'$, the problem \eqref{weakformu} has a unique solution $u\in \Wo$. This solution satisfies the a-priori estimate
\[
\|u\|_W \le c_s^{-1} \|f\|_{H'}.
\]
\end{theorem}
We consider two special cases in which the inf-sup stability constant $c_s$ can be shown to be bounded away from zero uniformly in $T$, cf. Remark~\ref{remstab}.
\begin{proposition} \label{thmST}
Assume that there is $c_0 >0$ such that
\begin{equation}\label{cond1}
  \alpha - \frac12 \DivG \bw \geq c_0 \quad \text{on}~~\Gs
\end{equation}
holds.
 Then the inf-sup property \eqref{infsup} holds with $c_s= \frac{\min \{\nu_d, c_0\}}{\sqrt{2} ({1+\nu_d + \alpha_\infty})^2}$.
\end{proposition}
\begin{proof}
We follow the arguments as in the proof of Lemma~\ref{la:infsup}. Instead of $v=u_\gamma$ we take $v=u$ as a test function. This yields
\begin{equation} \label{tth}
\la\dot{u},u\ra +  a(u,u) \ge  \hat{c}\|u\|_H^2, \quad \hat c:=\min\{\nu_d, c_0 \}.
\end{equation}
We set $v=z+\mu u$ ($z$ as in the proof of Lemma~\ref{la:infsup}). Taking $\mu:=\frac{c_1+1}{2 \hat c} \geq 1$ we get
\[
   \la\dot u,v\ra +  a(u,v) \ge \frac12\|u\|_W^2 \ge \frac{\sqrt{2}}{4} \mu^{-1}\|u\|_W\|v\|_H = \frac{\min \{\nu_d, c_0\}}{\sqrt{2} ({ 1+\nu_d + \alpha_\infty})^2}\|u\|_W\|v\|_H .
\]
This completes the proof.
\end{proof}
\medskip

As a second special case, we consider the diffusion equation  on a moving surface \eqref{transport1}.
A smooth solution to this problem satisfies the zero average condition \eqref{average}. Functions $u$ from $H$ satisfying
\eqref{average} obey the Friedrichs inequality
\begin{equation}\label{Fr}
\int_{\Gamma(t)} |\nabla_\Gamma u|^2 \, ds \geq c_F(t) \int_{\Gamma(t)}  u^2 \, ds\quad\text{for all}~t\in[0,T],
\end{equation}
with $c_F(t) >0$.
The Friedrichs inequality helps us to get additional control on the $L^2$ -norm of $u$ in proving the inf-sup inequality and so to improve the stability constant $c_s$.
Below we shall make more precise when a solution to  \eqref{weakformu} satisfies a zero average condition, and
how the stability estimate is improved.

We introduce the following
subspace of $\Wo$:
\[
\widetilde{W}_0:= \{\, u \in \Wo~|~ \int_{\Gamma(t)} u(\cdot, t) \, ds =0 \quad \text{for all}~~t \in [0,T]\,\}.
\]
\begin{proposition}\label{Prop2}
Assume $ \alpha ={\Div}_\Gamma\bw$, $f\in H'$ is such that $\la f,\phi\ra=0 $ for all  $\phi \in C_0^1(0,T)$.
 Then the solution $u \in \Wo$ of \eqref{weakformu} belongs to $\widetilde{W}_0$.
Additionally assume that there exists a $c_0 >0$ such that
\begin{equation} \label{ass7}
   \DivG \bw (x,t) +\nu_d c_F(t) \geq c_0 \quad \text{for all}~~x \in \Gamma(t),~t \in [0,T]
\end{equation}
holds.  Then the inf-sup property \eqref{infsup} holds, with $c_s= \frac{\min \{\nu_d, c_0\}}{2\sqrt{2} (1+\nu_d + \alpha_\infty)^2} $ and $\Wo$ replaced by the subspace $\widetilde{W}_0$.
\end{proposition}
\begin{proof} Let $u \in \Wo$ be the solution of \eqref{weakformu}. Define $U(t):= \int_{\Gamma(t)} u(\cdot,t) \, ds$. Using Theorem~\ref{PropW} (ii) we get
\[
  \int_0^T U(t)^2 \, dt \leq \int_0^T \|u_{|\Gamma(t)}\|_{L^2(\Gamma(t))}^2 |\Gamma(t)| \, dt \leq c \|u\|_W^2.
\]
Hence, $U \in L^2(0,T)$ holds. Take $\phi \in C_0^1(0,T)$, and thus $\phi \in C_0^1(\Gs)$. Note that
\begin{align*}
 - \int_0^T U(t) \phi'(t) \, dt &  =-\int_0^T  \int_{\Gamma(t)} u \phi'(t) \, ds  dt =-\int_0^T  \int_{\Gamma(t)} u \dot \phi \, ds  dt \\ &  =
\la \dot u, \phi \ra + (u, \phi \DivG \bw)_0 \leq c \|u\|_W \|\phi\|_H \leq c \|u\|_W \|\phi\|_{L^2(0,T)}.
\end{align*}
This implies that $U \in H^1(0,T)$ holds. Using $\la f,\phi\ra=0 $  for arbitrary $\phi \in C_0^1(0,T)$, we obtain:
\begin{equation*}
 \int_0^T U' \phi \, dt  = - \int_0^T U \phi' \, dt = \la \dot u, \phi \ra + (u, \phi \DivG \bw)_0  =\la \dot u, \phi \ra +a(u,\phi)= \la f,\phi\ra=0.
\end{equation*}
Thus, $U$ is a constant function. From $U(0)= \int_{\Gamma_0} u(\cdot,0) \, ds=0$ it follows that $U(t)=0$ for all $t \in [0,T]$ and thus $u  \in \widetilde{W}_0$ holds.

For the proof of the inf-sup property,  we follow the arguments as in the proof of Lemma~\ref{la:infsup}. Take $u \in \widetilde{W}_0$. Instead of $v=u_\gamma$ we take $v=u$ as a test function. Using the Friedrichs inequality \eqref{Fr} and the assumption in \eqref{ass7}, we get
\begin{align*}
\la\dot{u},u\ra +  a(u,u) &  \ge  \frac12 (\DivG \bw, u^2)_0 + \nu_d \|\nablaG u\|_0^2 \geq \frac12 ( \DivG \bw + \nu_d c_F, u^2)_0 + \frac12 \nu_d \|\nablaG u\|_0^2 \\ &  \geq \frac12 c_0 \|u\|_0^2 +\frac12 \nu_d \|\nablaG u\|_0^2 \geq  \hat{c}\|u\|_H^2, \quad \hat c:=\frac12 \min\{\nu_d, c_0 \}.
\end{align*}
 Take $v:=z+\mu u$ ($z$ as in the proof of Lemma~\ref{la:infsup}). Taking $\mu:=\frac{c_1+1}{2 \hat c} \geq 1$ we get
\[
   \la\dot u,v\ra +  a(u,v) \ge \frac12\|u\|_W^2 \ge \frac{\sqrt{2}}{4} \mu^{-1}\|u\|_W\|v\|_H = \frac{\min \{\nu_d, c_0\}}{2\sqrt{2} (1+(\nu_d + \alpha_\infty)^2)}\|u\|_W\|v\|_H .
\]
This completes the proof.
\end{proof}

\section{Time-discontinuous weak formulation} \label{ssectDG}\label{sec5}
In this section, we study a time-discontinuous  variant of the weak formulation in \eqref{weakformu}. This new variational formulation  is even weaker than \eqref{weakformu}. However, it can be seen as a time-stepping discretization  method for \eqref{weakformu} and is better suited for the discontinuous Galerkin discretization framework.

 Consider a  partitioning of the time interval:  $0=t_0 <t_1 < \ldots < t_N=T$, with a uniform time step $\Delta t = T/N$. The assumption of a uniform time step is made to simplify the presentation, but is not essential for the results derived below. A time interval is denoted by $I_n:=(t_{n-1},t_n]$.  The symbol $\Gs^n$ denotes the space-time interface corresponding to $I_n$, i.e.,  $\Gs^n:=\cup_{t \in I_n}\Gamma(t) \times \{t\}$, and $\Gs:= \cup_{1 \leq n \leq N}  \Gs^n $. We introduce the following subspaces of $H$:
\[H_n:=\{\, v \in H~|~v=0  \quad \text{on}~~\Gs \setminus \Gs^n\, \}.
\]
For $u \in H$ we use the notation $u_n:=u_{|\Gs^n} \in H_n$. Corresponding to the space  $H_n$, we define a material weak derivative as in Section~\ref{sectHW}. For $u \in H_n$:
\[
\la\dot u,\phi\ra_{I_n}:= - \int_{t_{n-1}}^{t_n} \int_{\Gamma(t)} u \dot \phi + u \phi\DivG\bw\, ds \, dt  \quad \text{for all}~~ \phi \in C_0^1(\Gs^n).
\]
If for $u\in H_n$ the norm
\[
\|\dot{u}\|_{H_n'}=\sup_{\phi\in C_0^1(\Gs^n)}\frac{\la\dot u,\phi\ra_{I_n}}{\|\phi\|_{H}}
\]
is bounded, then by a density argument, cf. Lemma~\ref{PropH},  $\dot{u}$ can be extended to a bounded linear functional on $H_n$.  We define the spaces
\[
  W_n= \{ \, v\in H_n~|~\dot v \in H_n' \,\}, \quad \|v\|_{W_n}^2 = \|v\|_{H}^2 +\|\dot v\|_{H_n'}^2.
\]
Finally, we define the broken space
\[
W^b:= \oplus_{n=1}^N W_n,~~\text{with norm}~~ \|v\|_{W^b}^2= \sum_{n=1}^N \|v\|_{W_n}^2= \|v\|_H^2+ \sum_{n=1}^N \|\dot v_n\|_{H_n'}^2 .
\]
Note the embeddings
$
  W \subset W^b \subset H
$, and furthermore:
\begin{equation} \label{idhulp}
  \la \dot u, v\ra = \la \dot u_n, v\ra_{I_n} \quad \text{for}~~u \in W, v \in H_n.
\end{equation}
For $u \in W_n$, we define the one-sided  limits
$u_+^{n}=u_+(\cdot,t_{n})$ and ${u}_{-}^n=u_{-}(\cdot,t_n)$
 based on \eqref{conti2}. The limits are well-defined in $L^2(\Gamma(t_n))$ thanks to Theorem~\ref{PropW} (item (iii)). At $t_0$ and $t_N$  only $u_+^{0}$ and $u_{-}^{N}$ are defined.
For $v \in W^b$, a jump operator  is defined by $[v]^n= v_+^n-v_{-}^n \in L^2(\Gamma(t_n))$, $n=1,\dots,N-1$. For $n=0$, we define $[v]^0=v_+^0$.

On the cross sections $\Gamma(t_n)$, $0 \leq n \leq N$,  of $\Gs$ the $L^2$ scalar product is denoted by
\[
 (\psi,\phi)_{t_n}:= \int_{\Gamma(t_n)} \psi \phi \, ds
\]
and for the corresponding norm we use the notation $\|\cdot\|_{t_n}$.
In addition to $a(\cdot,\cdot)$,  we define on the broken space $W^b$ the following bilinear forms:
\begin{align*}
 d(u,v)  =  \sum_{n=1}^N d^n(u,v), \quad d^n(u,v)=([u]^{n-1},v_+^{n-1})_{t_{n-1}},\quad
 \la \dot u ,v\ra_b =\sum_{n=1}^N  \la \dot u_n, v_n\ra_{I_n}
\end{align*}

Instead of \eqref{weakformu}, we now consider the following variational problem in the broken space: Given $f\in H'\subset(W^b)'$, find $u \in W^b$ such that
\begin{equation} \label{brokenweakformu}
  \la \dot u ,v\ra_b +a(u,v)+d(u,v) =\la f,v\ra \quad \text{for all}~~v \in W^b.
\end{equation}
This formulation is similar to a repeated application of the formulation \eqref{weakformu} on a sequence of time subintervals. The bilinear form $d(\cdot,\cdot)$ transfers information between neigboring time intervals in the usual DG weak sense.
We use $W^b$ for the test space, since the term $d(u,v)$ is not well-defined for an arbitrary $v\in H$.
 From an algorithmic point of view  this formulation has the advantage that due to the use of the broken space $W^b= \oplus_{n=1}^N W_n$ it can be solved in a time stepping manner. The final discrete method (Section~\ref{sec6}) is obtained by combining the variational formulation \eqref{brokenweakformu} and a Galerkin approach in which the space $W^b$  is replaced  by a finite element subspace. Before we turn to such a discretization method we first study \emph{consistency} and \emph{stability} of the weak formulation \eqref{brokenweakformu}.
\begin{lemma}\label{lem_aux1}
 Let $u \in \Wo$ be the solution of \eqref{weakformu}. Then $u$ solves the variational problem \eqref{brokenweakformu}.
\end{lemma}
\begin{proof}  Let $u \in \Wo$ be the solution of \eqref{weakformu}. Take a test function $v \in W_n$.
Since for  $u\in \Wo$ and for  any $t\in[0,T]$ the mapping \eqref{conti} is continuous as a mapping from a sufficiently small interval $[-\ep,0]$ or $[0,\ep]$  to $L^2(\Gamma(t))$,
it follows that $[u]^n=0$ for all $n=0,\dots,N-1$.  Hence we get $d(u,v)=0$. From \eqref{weakformu} and \eqref{idhulp} we get
\[
  \la \dot u_n,v\ra_{I_n} +a(u_n,v)= \la f,v\ra \quad \text{for all}~~v \in W^n.
\]
and thus $ \la \dot u ,v\ra_b +a(u,v)+d(u,v) =\la f,v\ra$ holds.
\end{proof}
\smallskip

For the stability analysis, we introduce the following norm on $W^b$:
\[
 \enorm{u}^2 := \|u\|_{W^b}^2 + \frac12  \|u_{-}^N\|_{T}^2 + \frac12 \sum_{n=1}^N  \|[u]^{n-1}\|_{t_{n-1}}^2, \quad u
\in W^b.
\]
Since the trace norms $\|u_{+}^{n-1}\|_{t^{n-1}}$ and $\|u_{-}^{n}\|_{t^{n}}$ are bounded by $\|u\|_{W_n}$,  the  norm
$\enorm{u}$ is equivalent to the $\|\cdot\|_{W^b}$ norm, but slightly more convenient for the analysis.
 Stability of the time-discontinuous formulation is based on the next lemma.

\begin{lemma} \label{thmstab} Set $\gamma:=2(\nu_d+\|\alpha-\frac12 \DivG \bw\|_{L^\infty(\Gs)})$.  The following inf-sup inequality holds:
 \begin{equation*}
 \inf_{0\neq u \in W^b}~\sup_{ 0\neq v \in W^b} \frac{
  \la \dot u, v\ra_b +d(u,v) +a(u,v)}{\enorm{u}\,\|v\|_H} \geq c_{\rm st}\,{ e^{-\gamma T}},
 \end{equation*}
 where $c_{\rm st}$ is independent of $T$ and $N$.
\end{lemma}
\begin{proof}
We follow the arguments as in the first part of the proof of Lemma~\ref{la:infsup}. Let $u = \sum_{n=1}^N u_n \in W^b$ be given. As a test function we use  $v=u_\gamma = e^{-\gamma t}u \in W^b$, and let $u_{\gamma,n}=e^{-\gamma t} u_n \in W_n$.   From the definition of the weak material derivative we get
\[
 \la \dot u_{\gamma,n},u_n\ra_{I_n}= \la \dot u_n, u_{\gamma,n} \ra_{I_n} - \gamma(u_n, u_{\gamma,n})_0,
\]
and using \eqref{partint} and  the choice of $\gamma$, this yields
\begin{equation*}  \begin{split}
  \la \dot u, u_{\gamma} \ra_b & =  \frac12 \sum_{n=1}^N \Big( \la \dot u_n, u_{\gamma,n} \ra_{I_n} + \la \dot u_{\gamma,n}, u_n \ra_{I_n} + \gamma (u_n, u_{\gamma,n})_0 \Big)  \\ & = \frac12 \sum_{n=1}^N \Big( (  u_{-}^n, u_{\gamma}^n )_{t_n} -  (  u_+^{n-1},u_{\gamma,+}^{n-1})_{t_{n-1}}\Big) -  \frac12( u, u_\gamma\DivG \bw)_0   + \frac12 \gamma(u, u_{\gamma})_0.
\end{split}
\end{equation*}
Setting $\gamma:=2(\nu_d+\|\alpha-\frac12 \DivG \bw\|_{L^\infty(\Gs)})$, we obtain
\begin{equation} \label{egk} \begin{split}
  \la \dot u, u_{\gamma} \ra_b +(\alpha u,u_ \gamma)_0& \ge \frac12 \sum_{n=1}^N \Big( e^{-\gamma t_n} \|u_{-}^n\|_{t_n}^2 - e^{-\gamma t_{n-1}} \|u_+^{n-1}\|_{t_{n-1}}^2 \Big) +\nu_d(u, u_{\gamma})_0\\
  &\ge \frac12 \sum_{n=1}^N \Big( e^{-\gamma t_n} \|u_{-}^n\|_{t_n}^2 - e^{-\gamma t_{n-1}} \|u_+^{n-1}\|_{t_{n-1}}^2 \Big)
  + \nu_d e^{-\gamma T} \|u\|_0^2.
\end{split}
\end{equation}
We also have (with $u^0:=0$):
\begin{equation} \label{kk}
 \begin{split}
 d(u,u_\gamma) & = \sum_{n=1}^N ([u]^{n-1}, u_{\gamma,+}^{n-1})_{t_{n-1}} = \sum_{n=1}^N e^{-\gamma t_{n-1}} ([u]^{n-1}, u_+^{n-1})_{t_{n-1}} \\
 & = \frac12 \sum_{n=1}^N e^{-\gamma t_{n-1}} \big( \|[u]^{n-1}\|_{t_{n-1}}^2 + \|u_+^{n-1}\|_{t_{n-1}}^2 - \|u^{n-1}\|_{t_{n-1}}^2 \big)  \\
 & = - \frac12 \sum_{n=1}^N \Big( e^{-\gamma t_n} \|u_{-}^n\|_{t_n}^2 - e^{-\gamma t_{n-1}} \|u_+^{n-1}\|_{t_{n-1}}^2 \Big)  \\ &\quad  + \frac12 e^{-\gamma T} \|u_{-}^N\|_T^2 + \frac12 \sum_{n=1}^N e^{-\gamma t_{n-1}} \|[u]^{n-1}\|_{t_{n-1}}^2.
\end{split}
\end{equation}
Combining the results in \eqref{egk}, \eqref{kk} with $(\nabla_\Gamma u,\nabla_\Gamma u_\gamma) \geq e^{-\gamma T} \|\nabla_\Gamma u\|_0^2$ we obtain
\begin{align*}
   & \la \dot u, u_\gamma \ra_b +d(u,u_\gamma) +a(u,u_\gamma) \\  & \geq \nu_d e^{-\gamma T} \|u\|_H^2 + \frac12 e^{-\gamma T} \|u_{-}^N\|_T^2 + \frac 12 \sum_{n=1}^Ne^{-\gamma t_{n-1}} \|[u]^{n-1}\|_{t_{n-1}}^2
\\
 &  \ge \min\{\nu_d,1\} { e^{-\gamma T}} \big( \|u\|_H^2 + \frac12  \|u_{-}^N\|_T^2 + \frac 12 \sum_{n=1}^N  \|[u]^{n-1}\|_{t_{n-1}}^2\big) .
\end{align*}
Note that $\|u_\gamma\|_H\le\|u\|_H \le\enorm{u}$ holds.

In remains to control the  $\sum_{n=1}^N \|\dot{u}_n\|_{H_n'}^2$ term in  $\enorm{u}^2$. From the definition of the $H_n'$ norm and the density of $C_0^1(\Gs^n)$ in $H_n$ it follows that for any $\ep \in (0,1)$, there exists $\phi_n \in C_0^1(\Gs^n)$ such that
\[
 \la \dot u_n, \phi_n \ra_{I_n}\ge (1-\ep) \|\dot{u}_n\|_{H_n'}\|\phi_n\|_{H}.
\]
We can scale $\phi_n$ such that $\|\phi_n\|_{H}=\|\dot{u}_n\|_{H_n'}$. Setting $z=\sum_{n=1}^N \phi_n$ we find
\[
\la \dot u, z \ra_b \ge (1-\ep)\sum_{n=1}^N \|\dot{u}_n\|_{H_n'}^2,\quad \|z\|_H^2= \sum_{n=1}^N \|\dot{u}_n\|_{H_n'}^2\le \enorm{u}^2.
\]
Furthermore, $d(u,z)=0$ and $a(u,z) \leq c \|u\|_H\|z\|_H$ hold. We let $v=z+\mu u_\gamma$ for sufficiently large $\mu$ and complete the proof using same arguments as in the proof of Lemma~\ref{la:infsup}.\quad
 \end{proof}
\medskip

We conclude that the weak formulations in
\eqref{weakformu} and  \eqref{brokenweakformu} are equivalent in the sense that the unique solution of the former is the unique solution of the latter.
As a direct consequence  of lemmas~\ref{lem_aux1} and~\ref{thmstab} we obtain the following theorem.

\begin{theorem}\label{th_br}
For any $f\in H'$, the problem \eqref{brokenweakformu} has a unique solution $u\in W^b$ satisfying the a priori estimate
\[
\enorm{u} \le c^{-1}_{\rm st}\,e^{\gamma T}  \|f\|_{H'},
\]
with $\gamma$ defined in Lemma~\ref{thmstab}.
\end{theorem}
\medskip

In the previous section, we noted that for the diffusion equation, which is a special special case of the surface equation \eqref{transport2}, the stability constant can be improved, cf.~Proposition~\ref{Prop2}. An analog holds  for  the discontinuous time-space problem \eqref{brokenweakformu}.

\begin{proposition}\label{Prop2b}
Assume $ \alpha ={\Div}_\Gamma\bw$, $f\in H'$ is such that $\la f,\phi\ra=0 $ for all  $\phi \in C_0^1(0,T)$,
 and there exists a $c_0 >0$ such that
\begin{equation*}
   \DivG \bw (x,t) +\nu_d c_F(t) \geq c_0 \quad \text{for all}~~x \in \Gamma(t),~t \in [0,T].
\end{equation*}
Then the unique solution $u\in W^b$ of the problem \eqref{brokenweakformu} satisfies the a priori estimate
\begin{equation} \label{est_br}
\enorm{u} \le 
c\,\|f\|_{H'},
\end{equation}
with some $c$ independent of $T$ and $N$.
\end{proposition}
\begin{proof} Thanks to assumptions and  Proposition~\ref{Prop2} we know that the solution $u$ to \eqref{weakformu} is from $\widetilde{W}_0$. Thus, $u$ satisfies
zero average condition~\eqref{average} and due to Lemma~\ref{lem_aux1} and Theorem~\ref{th_br} this is also the unique solution to \eqref{brokenweakformu}.
Therefore, we may make use of the Friedrichs inequality \eqref{Fr} and prove the a priori estimate \eqref{est_br} following the lines of the proofs of Lemma~\ref{thmstab} with $\gamma=0$.
\end{proof}

\section{Finite element method} \label{sectderivation}\label{sec6}

We introduce a conforming finite element method with respect to the time-discontinuous formulation
\eqref{brokenweakformu}. The method extends the Eulerian finite element approach from
\cite{OlshReusken08,OlshanskiiReusken08,ORX} and uses traces of volume space-time finite element
functions on $\Gs$ (the practical implementation uses a piecewise linear approximation of $\Gs$, as explained below).

To define our finite element space $W_{h,\Delta t}\subset W^b$, consider  the partitioning of the space-time volume domain $Q= \Omega \times (0,T] \subset \Bbb{R}^{d+1}$ into time slabs  $Q_n:= \Omega \times I_n$. Corresponding to each time interval $I_n:=(t_{n-1},t_n]$ we assume a given shape regular simplicial triangulation $\T_n$ of the spatial domain $\Omega$. The corresponding spatial mesh size parameter is denoted by $h$. For convenience we use a uniform time step $\Delta t=t_n-t_{n-1}$. Then
$\mathcal{Q}_h=\bigcup\limits_{n=1,\dots,N}\T_n\times I_n$ is a subdivision of $Q$ into space-time
prismatic nonintersecting elements. We shall call $\mathcal{Q}_h$ a space-time triangulation of $Q$.
Note that this triangulation is not necessarily fitted to the surface $\Gs$. We allow $\T_n$ to vary with $n$ (in practice, during time integration one may wish to adapt the space triangulation depending on the changing local geometric properties of the surface) and so
the elements of $\mathcal{Q}_h$ may not  match at $t=t_n$.

For any $n\in\{1,\dots,N\}$, let $V_n$ be the finite element space of continuous piecewise linear functions on $\T_n$. In Section~\ref{sec8}, we comment on the case of higher order finite element spaces.
First we define the volume time-space finite element space:
\[
V_{h,\Delta t}:= \{ \, v: Q \to \Bbb{R} ~|~  v(x,t)= \phi_0(x) + t \phi_1(x)~\text{on every}~Q_n,~\text{with}~\phi_0,\, \phi_1 \in V_n\,\}.
\]
Thus, $V_{h,\Delta t}$ is a space of piecewise bilinear functions with respect to  $\mathcal{Q}_h$, continuous in space and discontinuous in time. Now we define our surface finite element space as a space
of traces of functions from $V_{h,\Delta t}$ on $\Gs$:
\[
  W_{h,\Delta t} := \{ \, w:\Gs \to \Bbb{R}~|~ w=v_{|\Gs}, ~~v \in V_{h,\Delta t} \, \}.
\]
The finite element method reads: Given $f\in H'$, find $u_h \in W_{h,\Delta t}$ such that
\begin{equation} \label{brokenweakformu_h}
  \la \dot u_h ,v_h\ra_b +a(u_h,v_h)+d(u_h,v_h) =\la f,v_h\ra \quad \text{for all}~~v_h \in W_{h,\Delta t}.
\end{equation}
As usual in time-DG methods,  the initial condition for $u_h(\cdot,0)$ is treated in a weak sense and is  part of the $d(\cdot,\cdot)$-term.
Since, $u_h\in C^1(Q_n)$ for all $n=1,\dots,N$, the first term  in \eqref{brokenweakformu_h}
can be written as
\[
\la \dot u_h ,v_h\ra_b=\sum_{n=1}^N\int_{t_{n-1}}^{t_n}\int_{\Gamma(t)} (\frac{\partial u_h}{\partial t} +\bw\cdot\nabla u_h)v_h ds\,dt.
\]
Note that this formulation  allows one to solve the space-time problem in a time marching way, time slab by time slab.

At this moment, we have no proof of an inf-sup stability result for the fully discrete problem \eqref{brokenweakformu_h}, similar to the one in Lemma~\ref{thmstab}. A specific technical difficulty in adopting the proof of Lemma~\ref{thmstab} to the discrete case is that an exponentially scaled trial function $u_\gamma=e^{\gamma t} u_h$ does not belong to the space of test functions anymore. One way out would be to modify the discrete space of test functions accordingly. Then, the required
stability result easily follows, but this leads to a finite element method dependent on the parameter $\gamma$. This is not what we use in practice.

Clearly, in the special case of condition \eqref{cond1}, one can prove the inf-sup inequality in the discrete case by
following the arguments of   Lemma~\ref{thmstab} with $\gamma=0$ and using the condition \eqref{cond1} to control the $H$-norm
of a trial function, cf. Proposition~\ref{thmST}.
Since for $\gamma=0$ the test function is taken the same as the trial function, the inf-sup inequality becomes the ellipticity result
 \begin{multline} \label{infsup2}
 \la \dot u_h, u_h\ra_b {+a(u_h,u_h)+d(u_h,u_h)}\\  \geq \min\{1,c_0\}(\nu_d  \|u\|_H^2 + \frac12  \|u_{-}^N\|_{T}^2 + \frac12 \sum_{n=1}^N  \|[u]^{n-1}\|_{t_{n-1}}^2)\quad\text{for all}~u_h \in W_{h,\Delta t}.
 \end{multline}
The space $W_{h,\Delta t}$ has a finite dimension, and hence the ellipticity result \eqref{infsup2} is sufficient for existence of a unique
solution. We summarize this in the form of the following proposition.
\begin{proposition} Assume \eqref{cond1}, then for any $f\in H'$ the problem \eqref{brokenweakformu_h}
has a unique solution $u_h \in W_{h,\Delta t}$ satisfying the a priori estimate
\[
\left(\nu_d  \|u\|_H^2 + \frac12  \|u_{-}^N\|_{T}^2 + \frac12 \sum_{n=1}^N  \|[u]^{n-1}\|_{t_{n-1}}^2\right)^{\frac12} \le \nu_d^{-\frac12}\max\{1,c_0^{-1}\}  \|f\|_{H'}.
\]
\end{proposition}

The special case of the diffusion surface problem as described in propositions~\ref{Prop2} and ~\ref{Prop2b}
is less straightforward to handle, since in the discrete setting, the method is not pointwise conservative, i.e.
the zero average condition \eqref{average} can be satisfied only approximately. Stability analysis of the
finite element problem \eqref{brokenweakformu_h}  in this interesting case is presented in the recent report \cite{refARrecent}.

Before presenting numerical results, we comment on a few implementation aspects of our surface finite element
method. More details are found in the recent report \cite{refJoerg}.

By choosing the test functions $v_h$ in \eqref{brokenweakformu_h} per time slab, as in standard space-time DG methods, one obtains an implicit time stepping algorithm. Two implementation issues are the approximation of the space-time integrals in the bilinear form $\la \dot u_h ,v_h\ra_b +a(u_h,v_h)$ and the representation of the the finite element trace functions in $W_{h,\Delta t}$.  To approximate the integrals, we first make use of the
transformation formula \eqref{transform} converting space-time integrals to surface integrals over $\Gs$, and next we
approximate $\Gs$ by a `discrete' surface $\Gs^h$.  In our implementation, the approximate surface $\Gs^h$ is the zero level of $\phi_h\in W_{\hat h, \hat{\Delta t}}$,
where $\phi_h$ is the nodal interpolant of a level set function $\phi(x,t)$, the zero level of which is the surface $\Gs$. In the experiments considered in the next section, for $\phi$ we use the signed distance function of $\Gs$. To reduce the ``geometric error'', the interpolation $\phi_h \in W_{\hat h, \hat{\Delta t}}$ can be done in a finite element space  with mesh size $\hat h < h$, $\hat{\Delta t} < \Delta t$. In the the next section examples, we use $\hat h = \frac 12 h$, $\hat{\Delta t} = \frac12 \Delta t$ (one regular refinement of the given
outer space-time mesh).

For the representation of  the finite element  functions in $W_{h,\Delta t}$, we consider traces of the standard nodal basis functions
in the volume space $V_{h,\Delta t}$. Obviously, only nodal functions corresponding to elements $P\in\mathcal{Q}$ such that $P\cap\Gs^h\neq0$ should be taken into account.  In general, these trace functions form a frame in $W_{h,\Delta t}$.
A finite element surface solution is represented as a linear combination of the elements from this frame.
Linear systems resulting in every time step  may have more than one solution, but every solution yields the same trace function, which is the unique solution of \eqref{brokenweakformu_h}. In the numerical experiments below we used a direct solver for computing the discrete solution. Linear algebra aspects of the surface finite element method  have been addressed in \cite{OlshanskiiReusken08} and will be further investigated in future work.

\section{Numerical experiment}\label{sec7}
In this section, we  present results of a few numerical experiments to illustrate  properties of the space-time finite element method introduced in Section~\ref{sectderivation}.

{\bf{ Example 1}.} First, we consider a shrinking sphere, represented as the zero level $\Gamma(t)$ of  the level set function
$
 \phi(x,t)=x_1^2+x_2^2+x_3^2-1.5^2 e^{-t}.
$
The initial sphere has a radius of $1.5$ at $t=0$. The corresponding velocity field is given by
$
\mathbf{w}={ -{\frac34 e^{-t/2}}}\mathbf{n},
$
where $\mathbf{n}$ is the unit outward normal on $\Gamma(t)$. Hence ${\Div}_\Gamma \bw = -1$. Hence, the coefficient $\alpha$ is negative and the condition \eqref{cond1} is not satisfied.
We choose a solution  $u(x,t)=(1+x_1x_2x_3) e^{t}$ and  thus the  right-hand side is given by
$
f(x,t)=(-1.5 e^{t}+\frac{16}{3} e^{2t})x_1x_2x_3.
$
The problem is solved by the space-time DG method in the time interval $0\leq t\leq 1$.

The outer domain is chosen as $\Omega=[-2,2]^3$. The outer spatial triangulation $\mathcal T$ is a uniform tetrahedral triangulation of $\Omega$ with mesh size $h=h_k=2^{-k}$, $k=1,\ldots,4$. This triangulation is chosen independent of $n$.
 The time step is taken  as $\Delta t = \Delta t_{\ell}=2^{-\ell}$, $\ell=0, \ldots,4$. The outer space-time finite element space  $V_{h, \Delta t}$ and the induced trace space are defined as explained in the previous section.
The smooth space-time manifold $\Gs=\cup_{t \in [0,1]}\Gamma(t) \times \{t\} $ is approximated as follows. To a given outer space-time mesh one regular refinement (in space and time) is applied.
The approximate, piecewise affine,  surface $\Gs^h = \cup_{t \in [0,1]}\Gamma_h(t)$ is defined as the zero level set of the nodal interpolant  of
the level set function $\phi(t)$ on this  refined outer mesh. This approximate space-time manifold is constructed per time step. 

 For the computation of  discretization errors the continuous solution $u$ is extended by constant values in normal direction. This extension is denoted by $u^e$. The initial condition $u_h(t_0)$ is the trace of the nodal interpolant to $u^e(t_0)$.
We compute two approximate discrete errors as follows.
$$err_{L^2(t_N)}=\|u^e-u_h\|_{L^2(\Gamma_h(t_N))},$$
and
\begin{align*}
err_{L^2(H^1)} &=
\Big\{\frac{\Delta t}{2}\|\nabla_{\Gamma_h}(u^e-u^e)\|_{L^2(\Gamma_h(t_0))}^2 +\sum_{i=1}^{N-1}\Delta t\|\nabla_{\Gamma_h}(u^e-u_h)\|_{L^2(\Gamma_h(t_i))}^2
\\ & +\frac{\Delta t}{2}\|\nabla_{\Gamma_h}(u^e-u_h)\|_{L^2(\Gamma_h(t_N))}^2\Big\}^{1/2}.
\end{align*}
 Figure~\ref{fig:l2err} shows
the convergence behavior of  $L^2(\Gamma(t_N))$-error with respect to space (in the left subfigure) and time (in the right subfigure).
We observe that the  $L^2$ error is of order $O(h^2)$ in space and of $O(\Delta t^2)$ in time.
\begin{figure}[ht!]
 \centering
\subfigure[]{
    \includegraphics[width=2.1in]{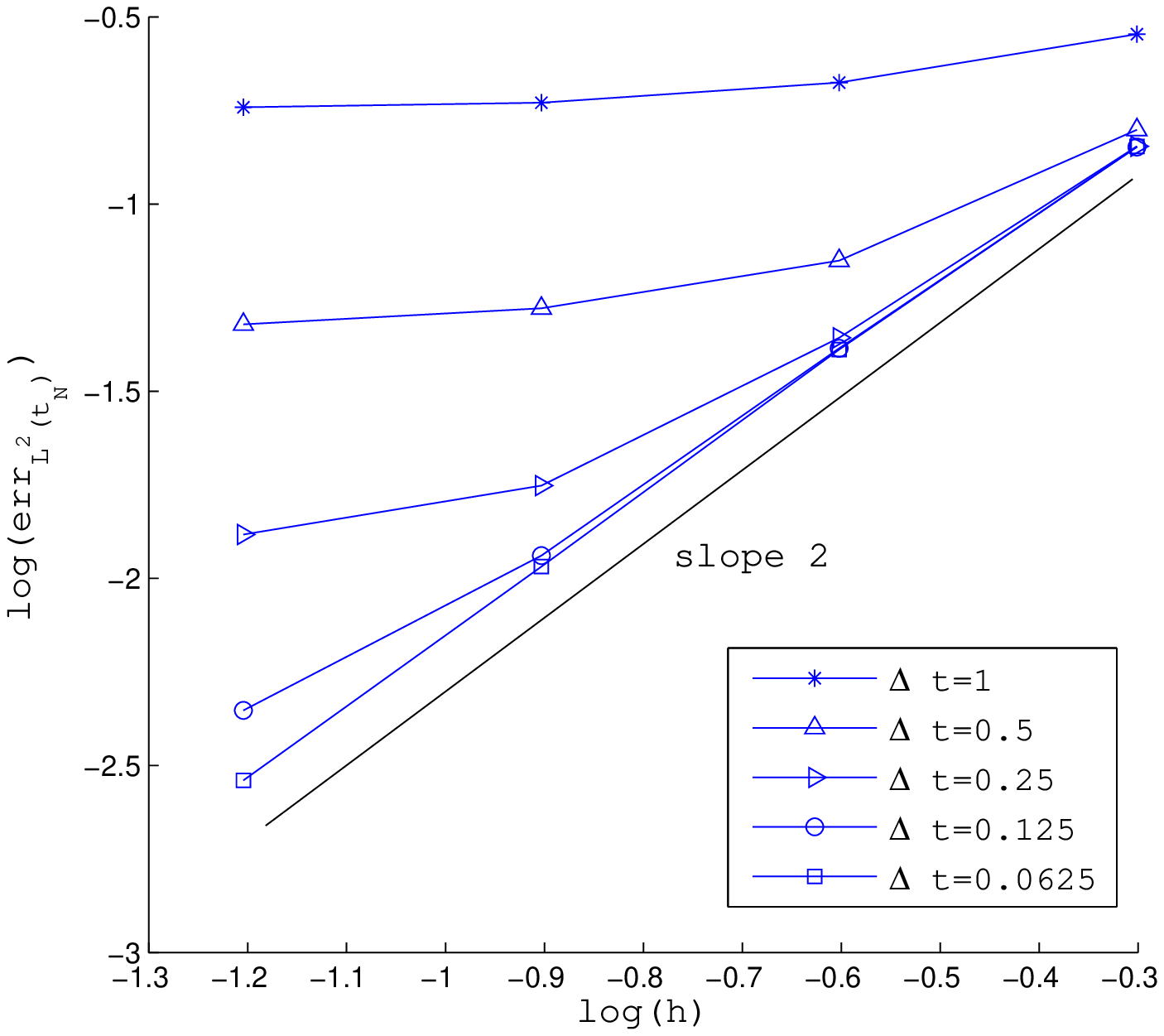}}
\subfigure[]{
    \includegraphics[width=2.1in]{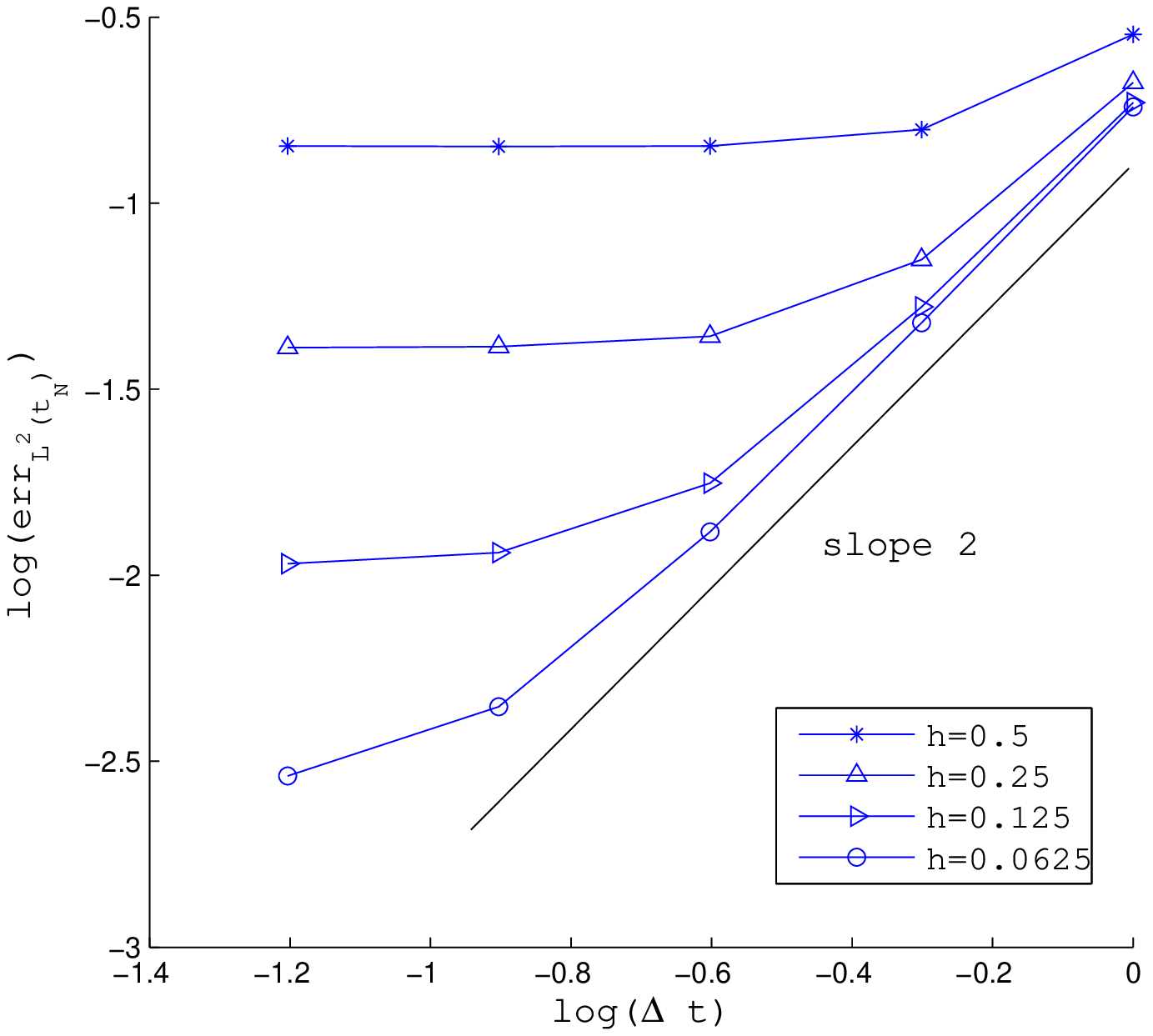}}
 \caption{Convergence w.r.t. $L^2(\Gamma(t_N))$ norm in Example 1.}
 \label{fig:l2err}
\end{figure}

 Figure~\ref{fig:h1err} shows
the convergence behavior of the $L^{2}(H^1)$-error with respect to space (in the left subfigure) and time (in the right subfigure).
 From the left subfigure it can be seen that the $L^{2}(H^1)$-error is of order $O(h)$ in space. The results in the right subfigure indicate that on the meshes used in this experiment
the $L^{2}(H^1)$-error is dominated by the space error.
Note that (very) large timesteps, even  $\Delta t=1$,  can  be used (even for small $h$), which indicates that the method has good stability properties.
\begin{figure}[ht!]
 \centering
\subfigure[]{
    \includegraphics[width=2.1in]{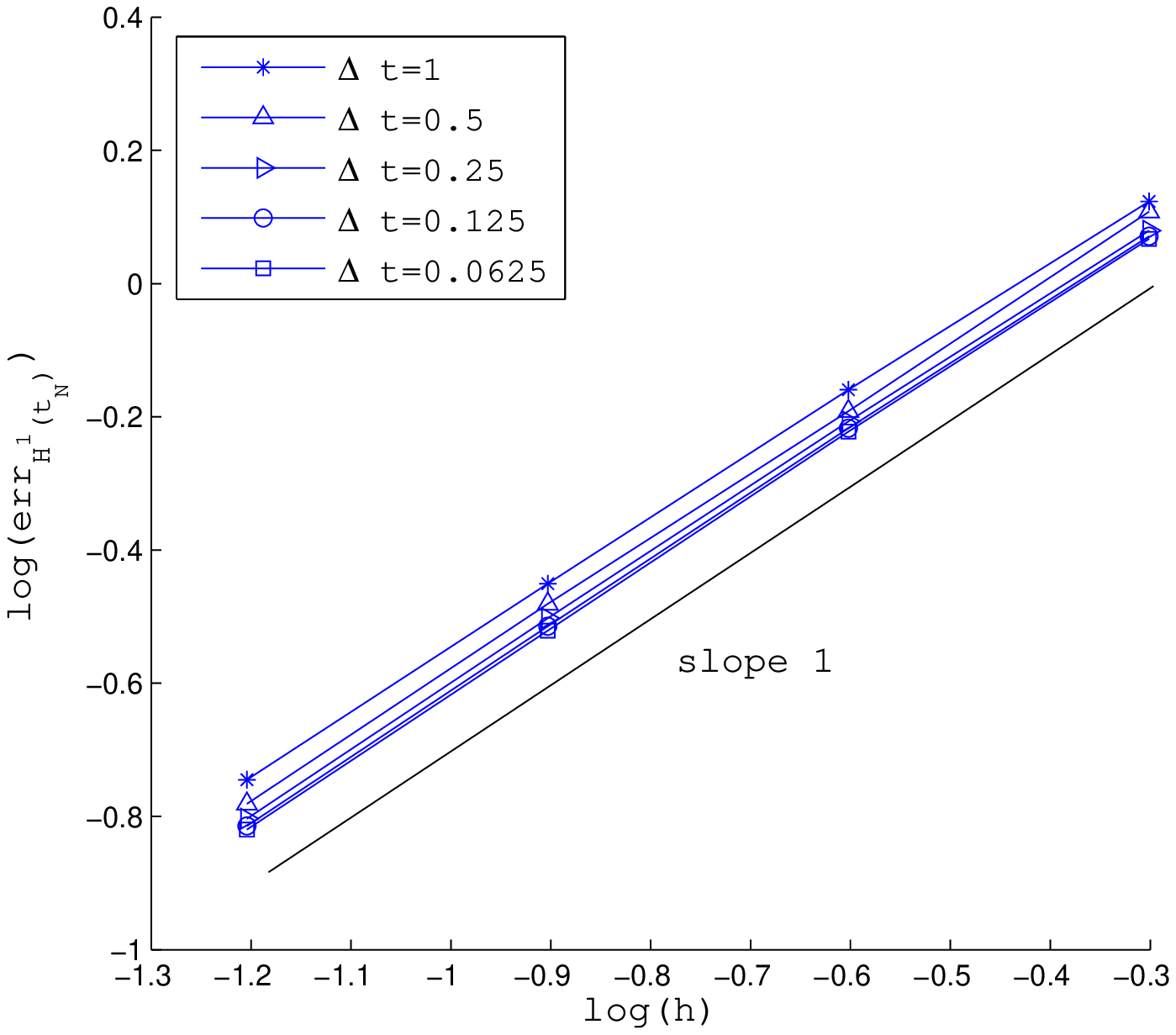}}
\subfigure[]{
    \includegraphics[width=2.1in]{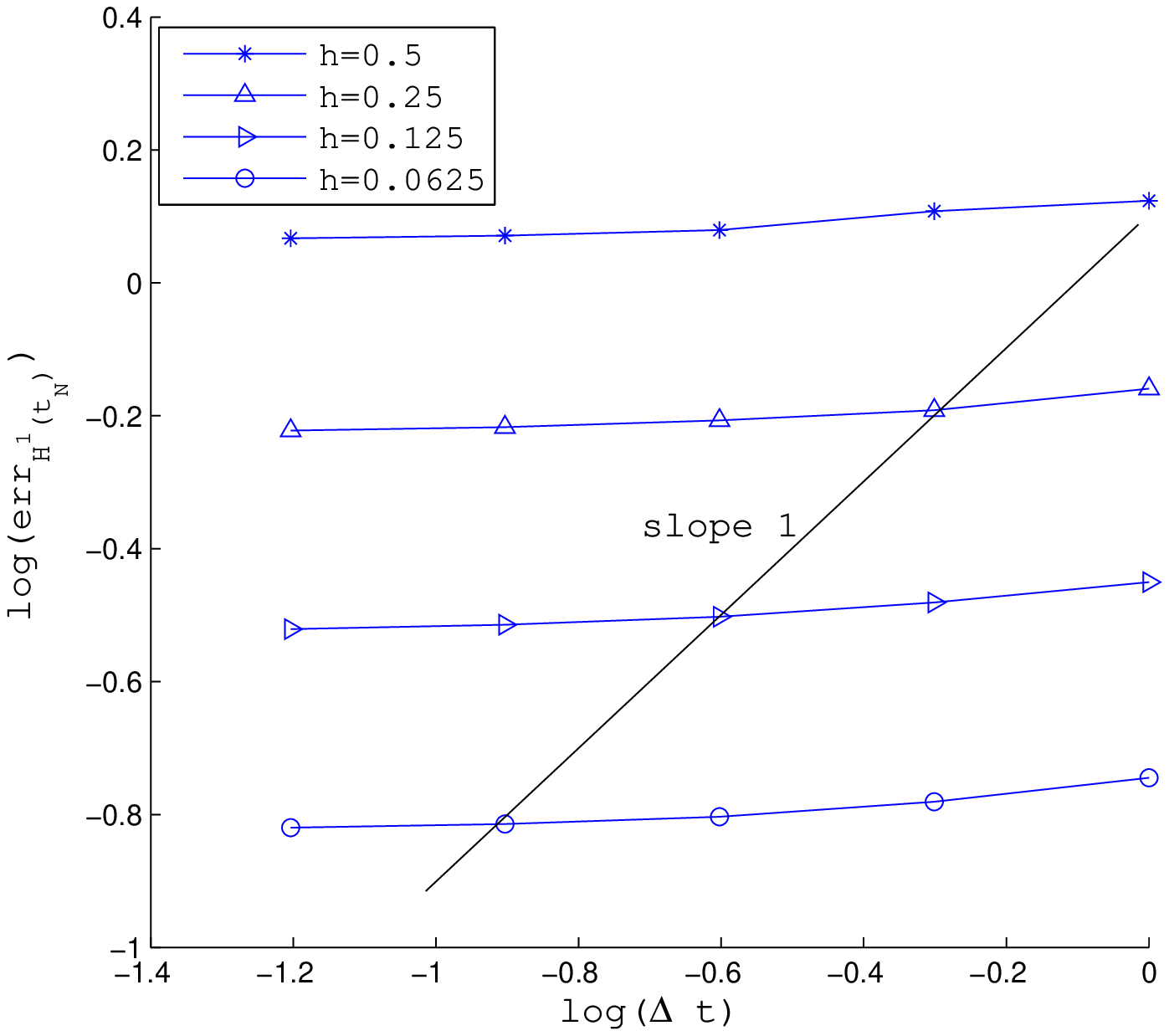}}

 \caption{Convergence w.r.t. $L^{2}(H^1)$ norm in Example 1.}
 \label{fig:h1err}
\end{figure}

To illustrate the convergence behavior of $H^1$-errors with respect to time, we  consider an {experiment} on the shrinking sphere, where the
 solution is given by
$u(x,t)=e^t,$
i.e. the function has maximal smoothness w.r.t. the spatial variable. A simple computation yields
$f=\frac{\partial u}{\partial t}+\bw\cdot\nablaG u+(\DivG \bw)u+\Delta_{\Gamma}u=\frac{\partial u}{\partial t}-u=0.$
 The $L^2(H^1)$-errors for this example are shown in Figure~\ref{fig:h1err_1}. From the left subfigure, we see that again
the $L^2(H^1)$-error is of order $O(h)$ in space, just as in the previous case. In the right subfigure we observe that  the error converges
in time with order $O(\Delta t)$.
\begin{figure}[ht!]
 \centering
\subfigure[]{
    \includegraphics[width=2.1in]{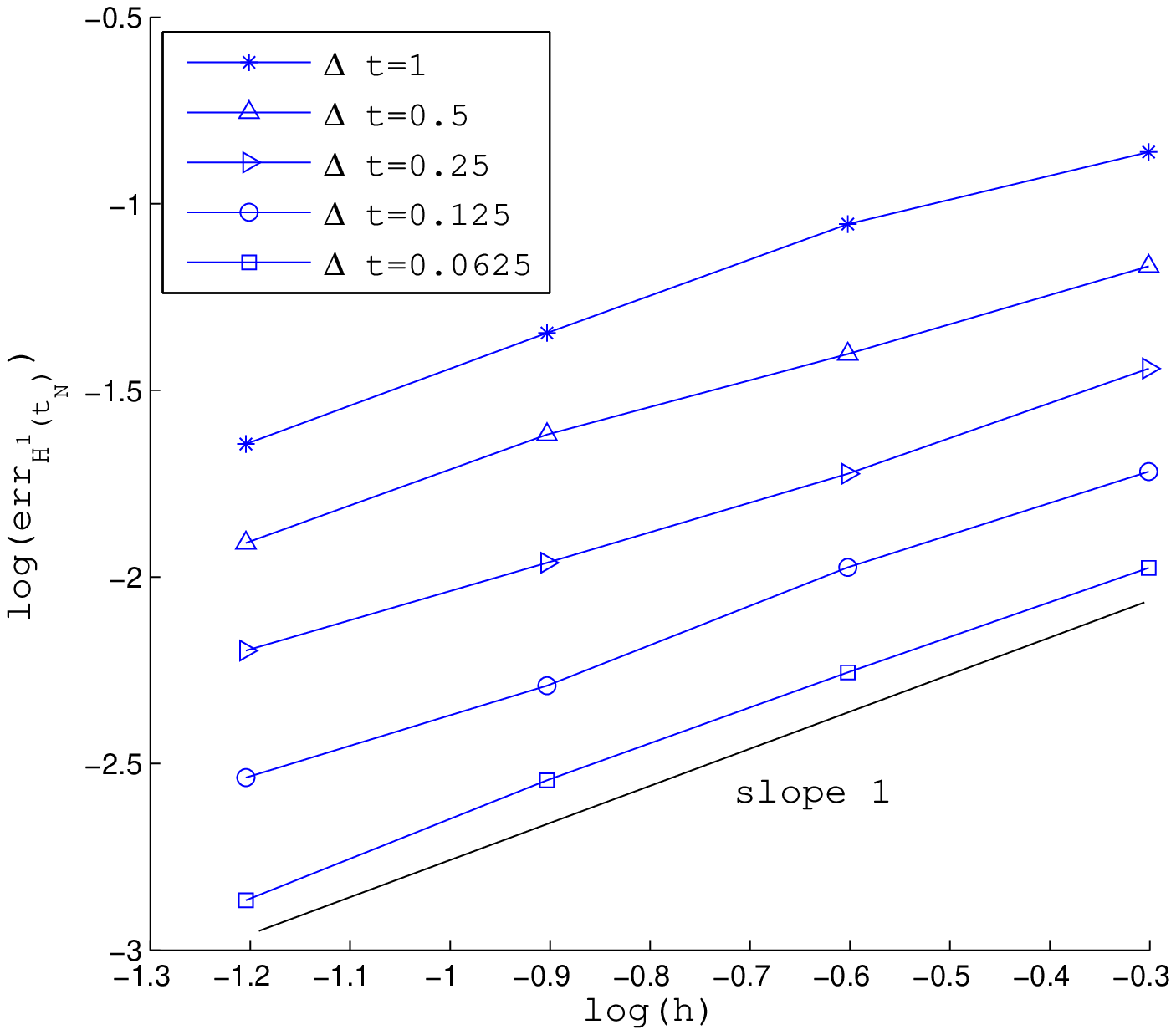}}
\subfigure[]{
    \includegraphics[width=2.1in]{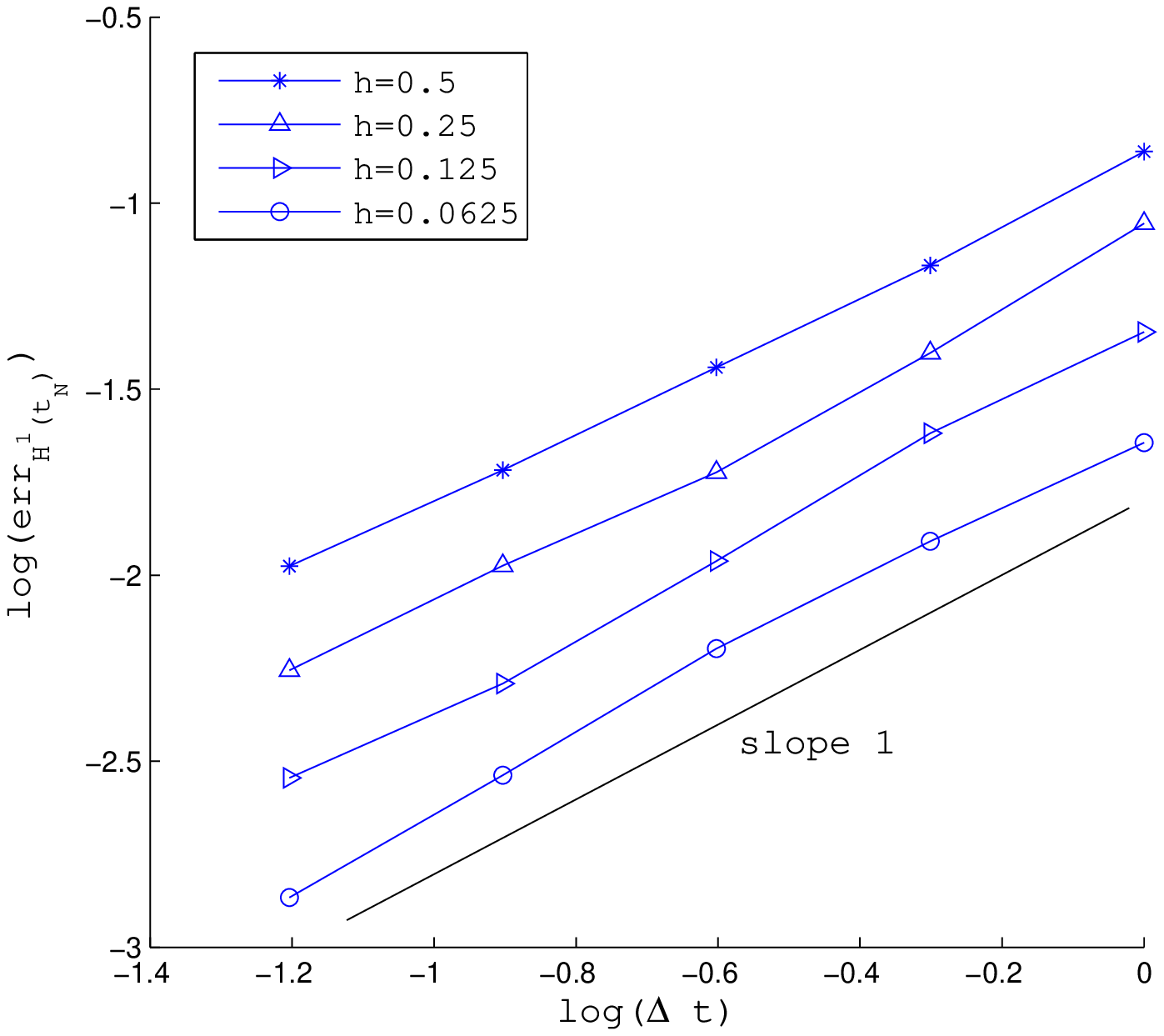}}

 \caption{Convergence w.r.t. $L^{2}(H^1)$ norm for $u(x,t)=e^t$.}
 \label{fig:h1err_1}
\end{figure}

{\bf Example 2.}  In this example, we consider a surface  diffusion problem as in \eqref{transport}  on a moving manifold. The initial manifold is given ({as in \cite{Dziuk88}}) by
$
\Gamma(0)=\{ \, x\in \Bbb{R}^3~|~ (x_1 -x_3^2)^2+x_2^2+x_3^2=1\, \}.
$
The velocity field that transports the surface is
$$\mathbf{w}(x,t)=\big(0.1x_1 \cos(t),0.2x_2 \sin(t),0.2x_3\cos(t)\big)^T.$$
 The initial concentration $u_0$ is chosen as
$u_0(x)=1+x_1 x_2 x_3.$

We set $\Delta t=0.1$ and compute the problem until $T=8$. The mesh size of the spatial outer mesh is $h=0.125$. An approximate surface $\Gs^h$ is constructed in the same way as in Example 1.
 In Figure~\ref{fig:dziuk} we show  the (aproximated) manifold and the discrete solution $u_h$ for different points in time.
\begin{figure}[ht!]
 \centering
 \begin{overpic}[width=0.45\textwidth]{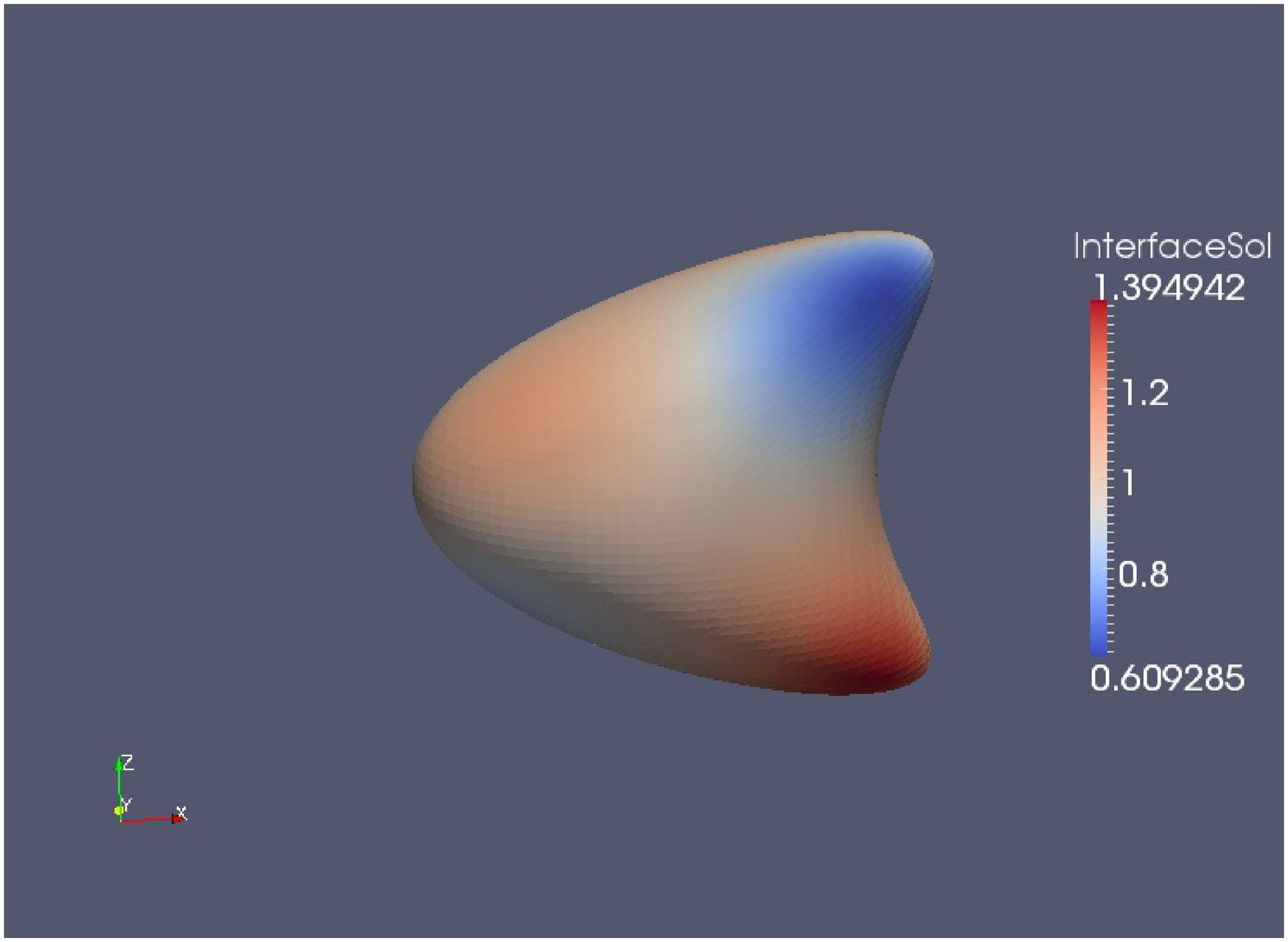}
  \end{overpic}
   \begin{overpic}[width=0.45\textwidth]{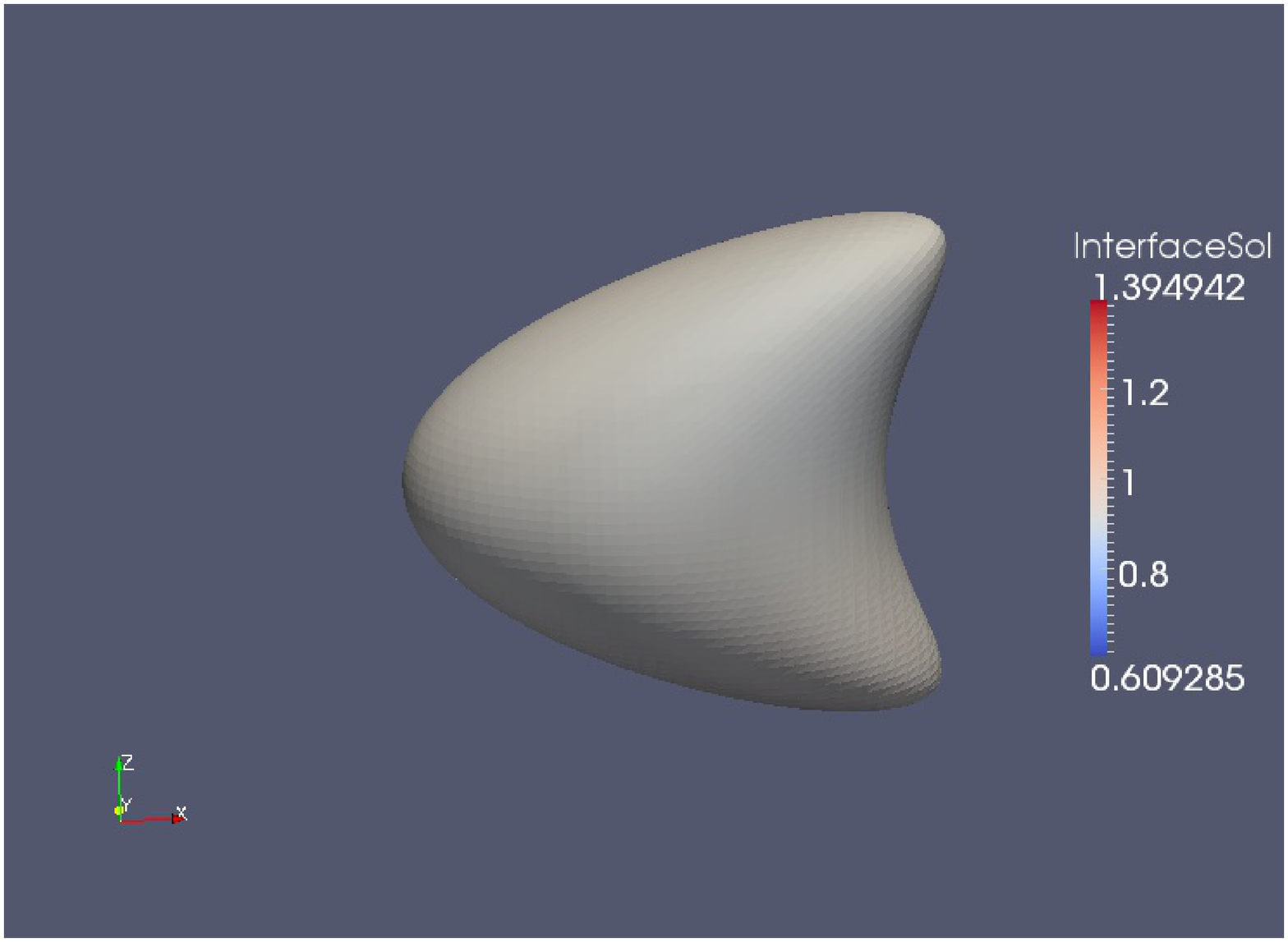}
  \end{overpic}
   \begin{overpic}[width=0.45\textwidth]{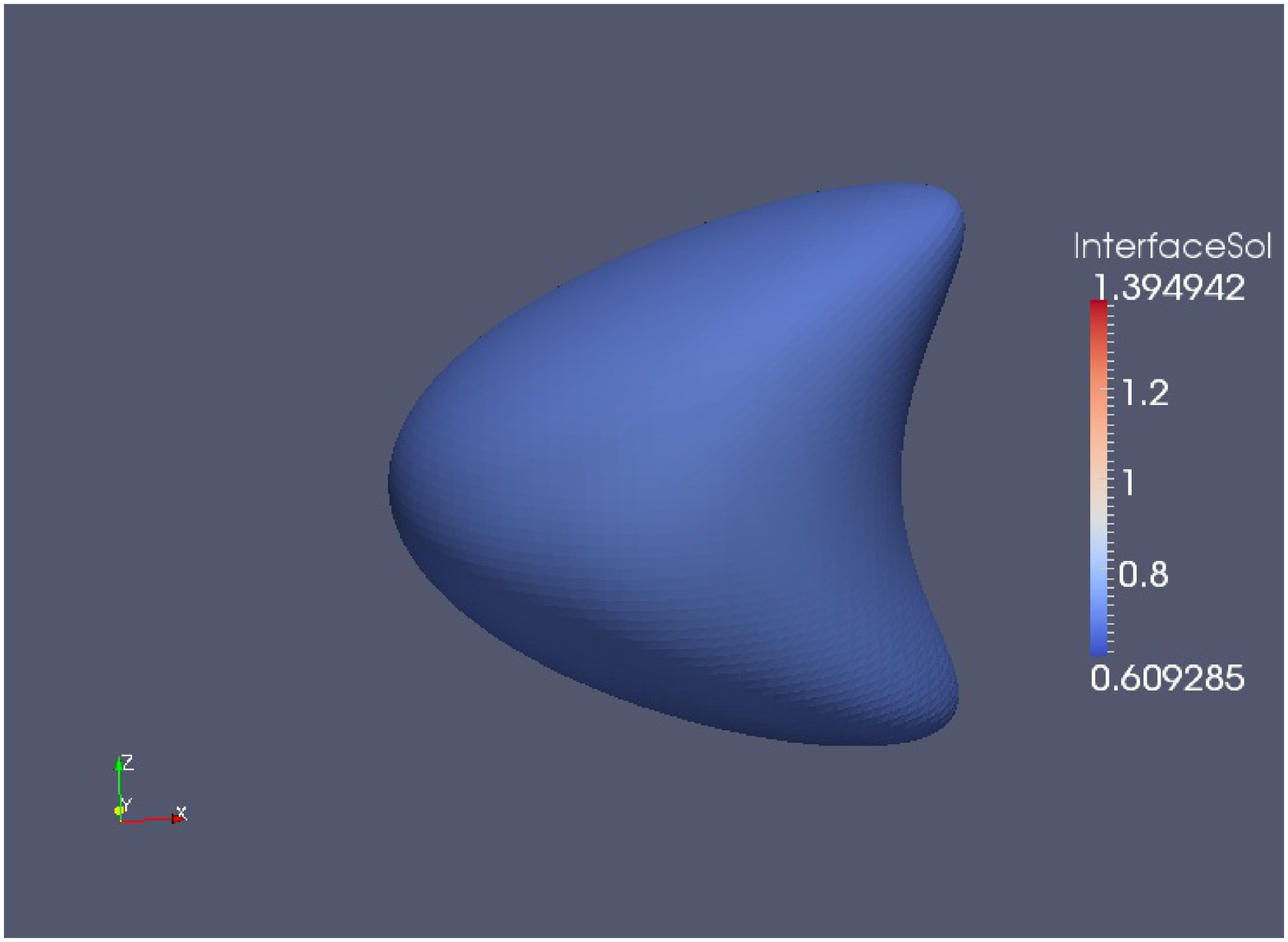}
  \end{overpic}
   \begin{overpic}[width=0.45\textwidth]{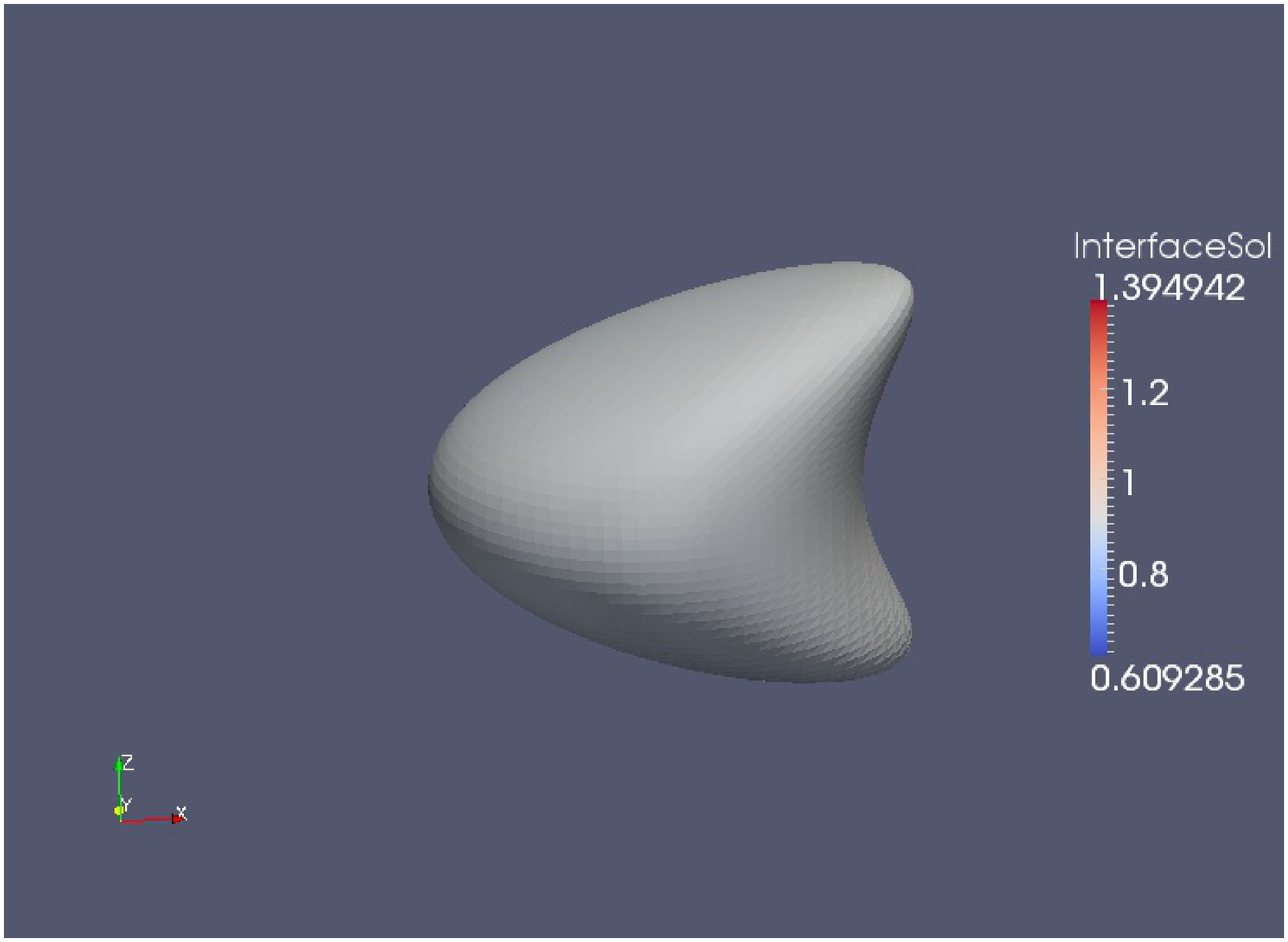}
  \end{overpic}
   \begin{overpic}[width=0.45\textwidth]{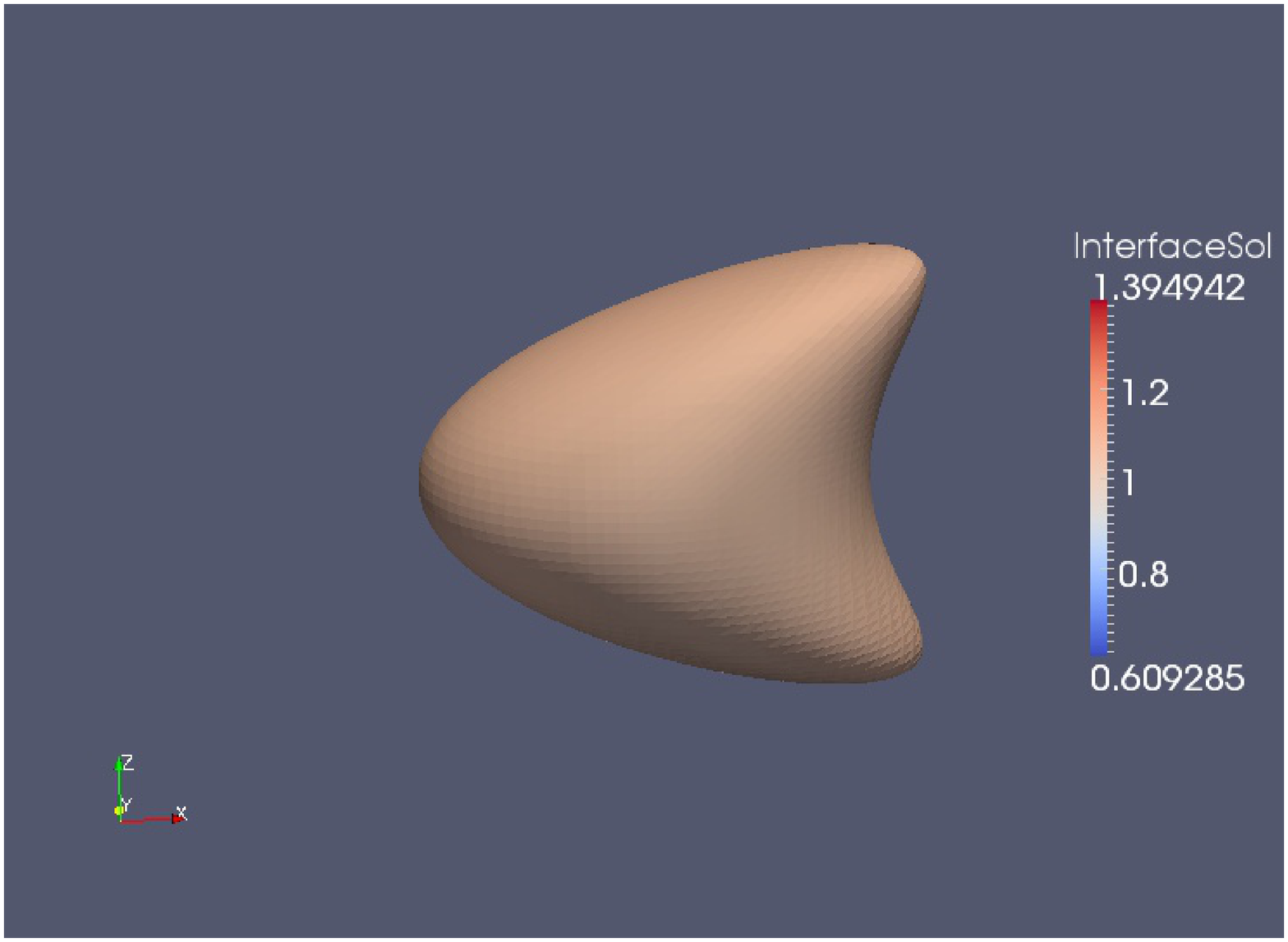}
  \end{overpic}
   \begin{overpic}[width=0.45\textwidth]{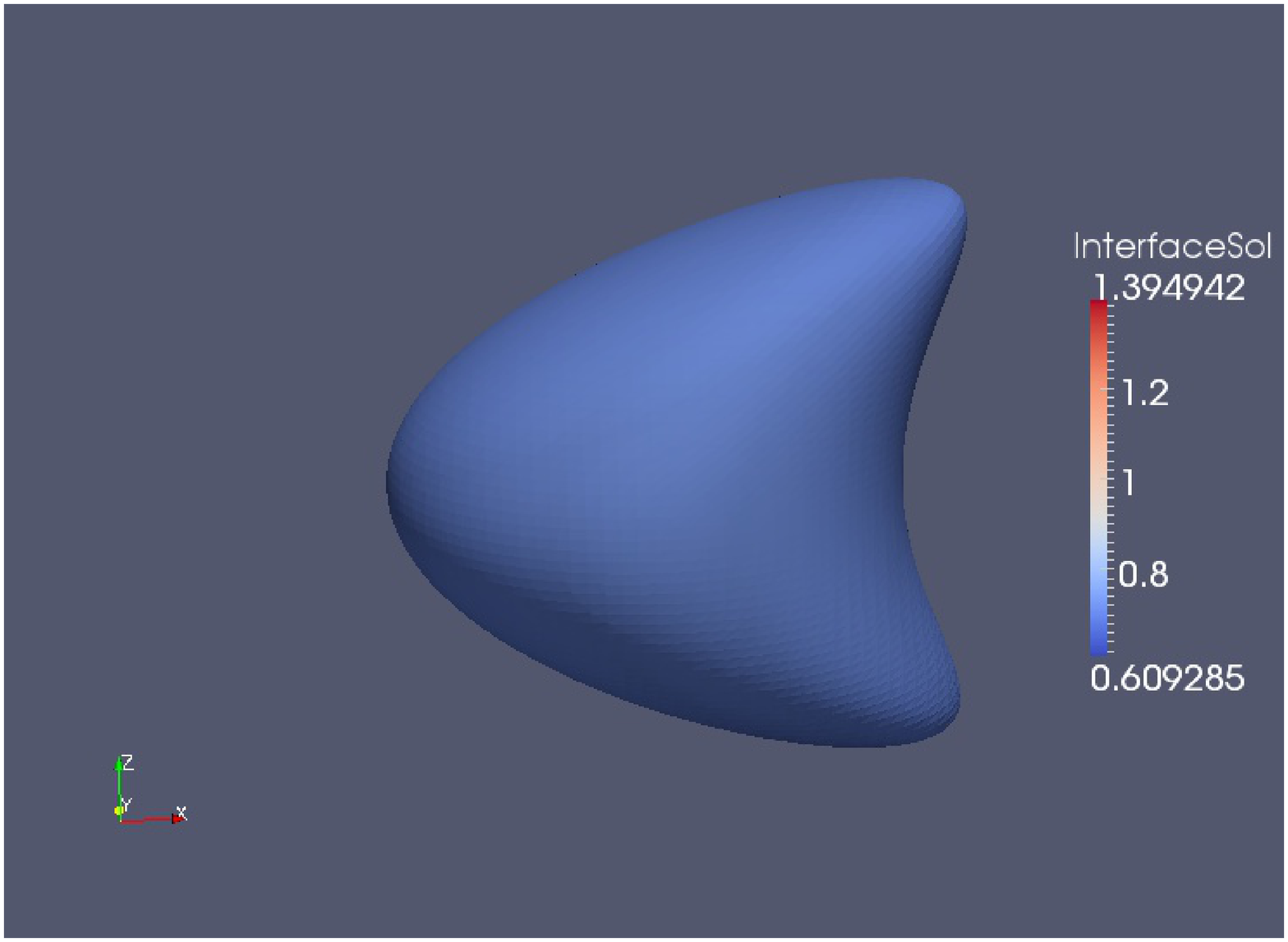}
  \end{overpic}
 \caption{Example~2: solutions for $t=0,0.4,2,4,6,8$.}\label{fig:dziuk}
\end{figure}

In this problem the total mass
$
M(t)=\int_{\Gamma(t)}u(\cdot,t) \, ds
$
is conserved  and equal to
$M(0)=|\Gamma(0)|\approx 13.6083.$
 We check how well the discrete analogon $
M_h(t)=\int_{\Gamma_h(t)}u_h(\cdot, t) \, ds
$ is conserved.
 In Figure~\ref{fig:mass:a}, for $t \in [0,T]$ this quantity is illustrated for different mesh sizes $h$ and a fixed time step $\Delta t=0.1$.
    In Figure~\ref{fig:mass:b},  the quantity is shown for different time steps $\Delta t$ and a fixed mesh size $h=0.125$.
  If we compute the average discrete mass $M_h:= \frac{1}{m} \sum_{j=1}^m M_h(t_j)$, with $t_1, \ldots, t_m$ the discrete time points (as shown in Fig.~\ref{figmass}) we obtain
for the absolute error $|M(0)-M_h|$ the numbers 0.2302, 0.0562, 0.0129, 0.0020 (corresponding to Fig.~\ref{fig:mass:a}) and 0.4973, 0.1126, 0.0208, 0.0052 (corresponding to Fig.~\ref{fig:mass:b}).
These results indicate that the method has a satisfactory discrete mass conservation property with a rate of convergence that is second order both with respect to $h$ and $\Delta t$.

Finally we note that the errors in the discretization are  caused not only by the space-time finite element discretization
but also by the geometric errors caused by the approximation of $\Gs$ by $\Gs^h$.

 \begin{figure}[ht!]
 \centering
\subfigure[]{\includegraphics[width=0.45\textwidth]{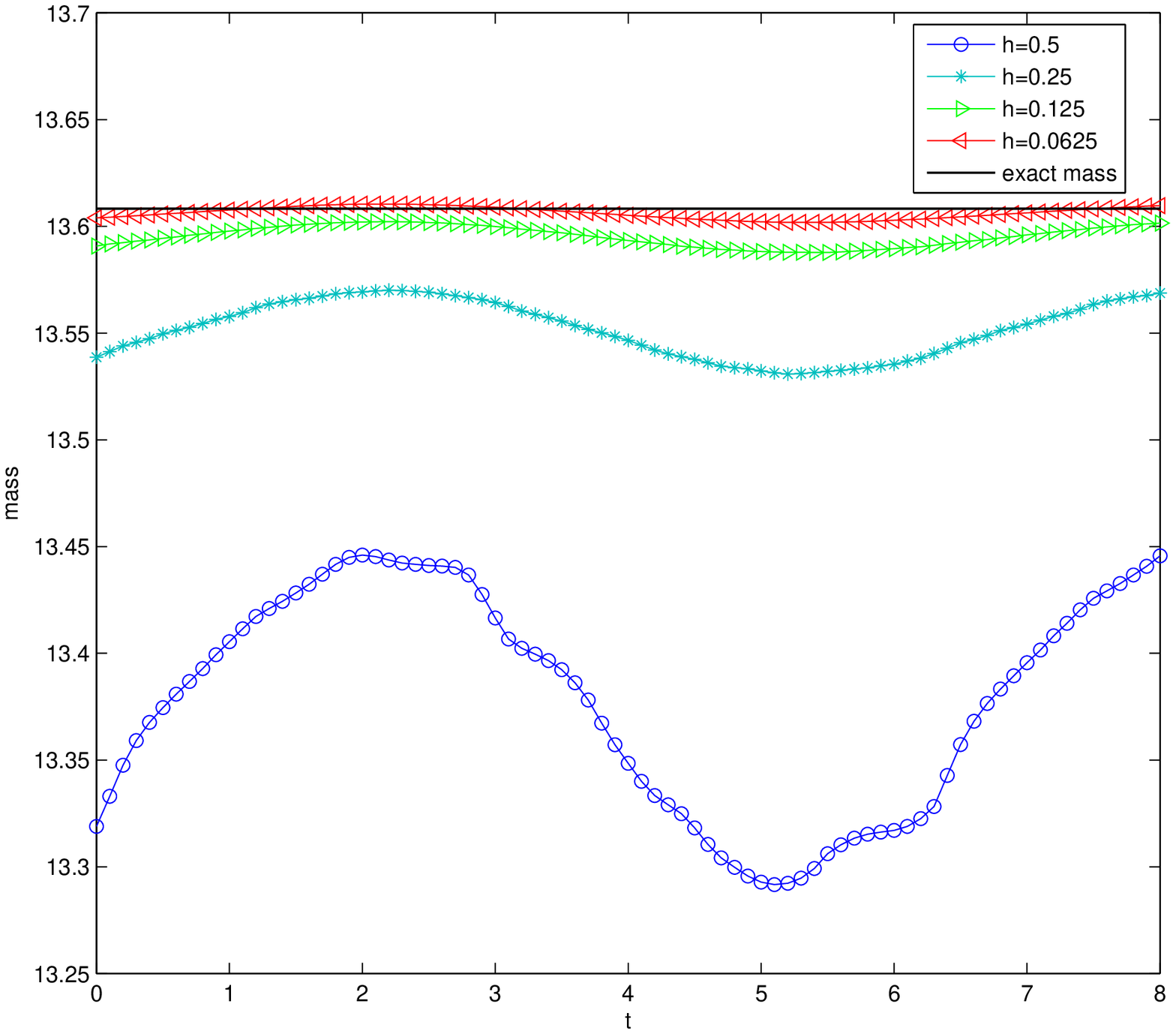}
\label{fig:mass:a}}
\subfigure[]{\includegraphics[width=0.45\textwidth]{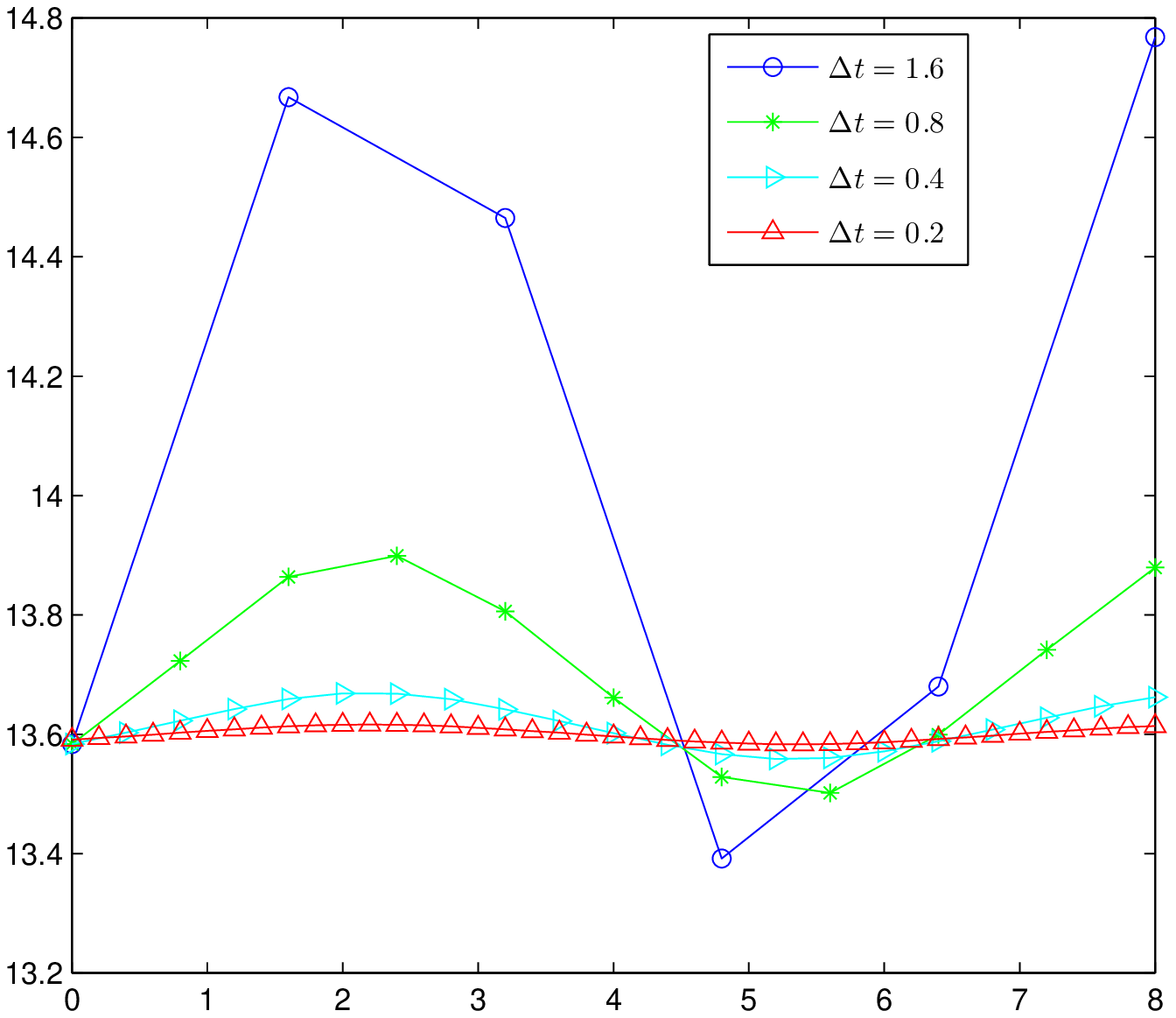}
\label{fig:mass:b}}
 \caption{Example~2: Discrete mass conservation.} \label{figmass}
\end{figure}

\section{Conclusions}\label{sec8}
In this paper we develop a mathematical framework for a new Eulerian finite element method for parabolic equations posed on evolving surfaces. The discretization method uses space-time elements.
 The space-time finite element method naturally relies on a space-time weak formulation. Such a formulation is
introduced and shown to be well-posed. The analysis  uses a smooth diffeomorphism between the space-time manifold and a reference domain.     This theoretical framework  does not allow to treat surfaces that undergo topological changes such as merging or splitting. The numerical method, however, can be applied in such situations.
Stabiliy of the discrete method is derived only for a  special case. Numerical experiments demonstrate stable behaviour and optimal convergence results in a more general setting. Extension of the finite element error analysis to more general problems is a topic of current research. In this paper, we consider only the case of piecewise linear (in space and time) finite elements. The method, however, is directly applicable with higher order finite elements. To benefit from the higher order approximation one needs a sufficiently accurate approximation of the continuous space-time manifold.

\subsection*{Acknowledgement}
Funding by the National Science Foundation through the project DMS-1315993,   the German Science Foundation (DFG) through the project RE 1461/4-1, and by the Russian Foundation for Basic Research through the projects 12-01-91330, 12-01-33084 is acknowledged. The authors   thank J\"org Grande for his contribution to the implementation of the finite element method.
\bibliographystyle{siam}
\bibliography{literatur}

\end{document}